\documentclass[11pt]{article}

\usepackage[utf8]{inputenc}
\usepackage[english]{babel}
\usepackage[T1]{fontenc}
\usepackage{amsfonts}
\usepackage{amsmath}
\usepackage{amssymb}
\usepackage{amsthm}
\usepackage{graphicx}
\usepackage{pstricks}
\usepackage{xcolor}
\usepackage{enumitem}

\definecolor{red}{rgb}{1.0,0.0,0.0}

\definecolor{blu}{rgb}{0.0,0.0,1.0}

\definecolor{gre}{rgb}{0.03,0.50,0.03}

\usepackage{thmtools}
\usepackage[pdfdisplaydoctitle,colorlinks,breaklinks,urlcolor=blue,linkcolor=blue,citecolor=blue]{hyperref}

\newtheorem{lemma}{lemma}[section]
\newtheorem{theorem}[lemma]{Theorem}
\newtheorem{prop}[lemma]{Proposition}
\newtheorem{cor}[lemma]{Corollary}
\newtheorem{defn}[lemma]{Definition}
\newtheorem{prob}{Problem}
\newtheorem{rem}[lemma]{Remark}
\newtheorem{hypo}[lemma]{Hypothesis}

\newcommand{\Hh}{\mathcal{H}}
\newcommand{\R}{\mathbb{R}}

\def\<{\left\langle }
\def\>{\right\rangle }

\newhsbcolor{mygreen}{.3 1 .6}

\def\H{\mathbb H}
\def\R{\mathbb R}
\def\N{\mathbb N}

\def\E{\mathbb E}
\def\P{\mathbb P}

\def\1{\mathbf 1}
\def\to{\rightarrow}

\begin{document}
%\title{\bf Optimal portfolio choice with path dependent labor income
%\textcolor{blue}{and deterministic time to retirement.}

\title{\bf Wage rigidity and retirement in portfolio choice}
%Labor income rigidity and retirement in portfolio selection
%Portfolio selection with labor income memory and retirement
 %Finite retirement time}
\author{
	Sara Biagini\footnote{Biagini(\texttt{sbiagini@luiss.it})
		is at the
		Dipartimento di Economia e Finanza,
		LUISS University, Rome, Italy.}
\and Enrico Biffis\footnote{Biffis (\texttt{e.biffis@imperial.ac.uk}) is at the Department of Finance, Imperial College Business School, London SW7 2AZ, UK.}
\and Fausto Gozzi%$^{\dagger}$
\footnote{Gozzi (\texttt{fgozzi@luiss.it})
% and Prosdocimi (\texttt{c.prosdocimi@luiss.it}) are
is at the
 Dipartimento di Economia e Finanza,
 LUISS University, Rome, Italy.}
\and
Margherita Zanella\footnote{Zanella (\texttt{margherita.zanella@polimi.it})
is at the
 Dipartimento di Matematica, Politecnico di Milano, Italy.}
 }

\date{\today}

\maketitle

\begin{abstract}                          % Abstract of not more than 200 words.
We study an agent's lifecycle portfolio choice problem with stochastic labor income, borrowing constraints and a finite retirement date. Similarly to \cite{BGP}, wages evolve in a path-dependent way, but the presence of a finite retirement time leads to a novel,  two-stage infinite dimensional stochastic optimal control problem with  explicit  optimal controls in feedback form.
This is possible as we find an explicit solution to the associated Hamilton-Jacobi-Bellman (HJB) equation, which is an infinite dimensional PDE of parabolic type.
The identification of the optimal feedbacks is delicate
due to
the presence of
time-dependent state constraints,  which appear to be new in the infinite dimensional stochastic control literature.
The explicit solution allows us to study the properties of optimal strategies and discuss their implications for portfolio choice.
As opposed to models with Markovian dynamics, path dependency can now modulate the hedging demand arising from the implicit holding of risky assets in human capital, leading to richer asset allocation predictions consistent
with wage rigidity and the agents learning about their earning potential.
\end{abstract}

\section{Introduction}
  We study the lifecycle portfolio problem of an agent receiving stochastic labor income and facing borrowing constraints as well as a finite retirement time.  Labor income is affected by economic shocks with some delay, which may characterize the stickyness of wages documented at least since \cite{KEYNES}.
Several studies have shown how discrete time Autoregressive Moving Average Processes (ARMA) can match the empirical evidence on earnings in a number of settings  (e.g.,
 	\cite{MaCurdy_1982},\cite{HUBBARD_SKINNER_ZELDES_1995},
 	\cite{MOFFITT_GOTTSCHALK_2002},
 	\cite{MEGHIR_PISTAFERRI_2004}), but the introduction of path dependency  in continuous time portfolio choice problems has traditionally been very challenging and analytical solutions have remained elusive (e.g., \cite{BENZONI_ET_AL_2007,DYBVIG_LIU_JET_2010}).
 	Contributions \cite{BGPZ,BGP} exploited the fact that under certain conditions Stochastic Delay Differential Equations (SDDEs)  can be understood as the weak limits of discrete time ARMA processes (e.g., \cite{DUNSMUIR_GOLDYS_TRAN_2016,LORENZ_2006,REISS_2002,TranPHD}) to introduce realistic labor income features in continuous time portfolio choice problems while retaining analytical tractability.  Similarly, \cite{YCW2022,YW2022} used SDDEs to model  cointegrated asset dynamics  and succeeded in solving continuous time mean-variance problems with and without pre-commitment, respectively.

\smallskip

The most tractable SDDEs are the linear ones, which are used by \cite{BGP} (see also \cite{DGZZ}) to model log-labor income $y$ with a delay term appearing in the drift.  The delay term weighs the past path of $y$ over a bounded time window $[-d, 0]$ with a Lebesgue-square integrable  function $\phi:[-d,0]\rightarrow \mathbb{R}$ (see \cite{BGPZ} for the case of a delay appearing also in the volatility term). As the agent can borrow against future labor income,  the budget constraint requires that total wealth must be non-negative, where total wealth is the sum of financial wealth and the current market value of future labor income (human capital).
Although contribution  \cite{BGPZ} does not focus on optimization, it provides an explicit formula for human capital when it is driven by an SDDE, thus opening the way to solving explicitly the infinite horizon portfolio problem studied in \cite{BGP}. % {\mygreen
There, the authors consider power utility and solve an optimal portfolio and consumption choice problem in which agents are allowed to borrow against future labor income, as in  \cite{DYBVIG_LIU_JET_2010}. The solution strategy follows the dynamic programming approach. As such an approach applies only in a Markovian setting, the key step is to rewrite the problem extending the state space
to include both current wealth/labor income and the past trajectory of labor income as state variables.
This brings to study an infinite dimensional HJB equation, which can be seen as an infinite dimensional version of the classical HJB for the Merton problem with labor income.
 A similar approach is used in \cite{BGZ}, which does not consider retirement but allows the weight appearing in the delay term to be a general Radon measure taking values in an uncertainty set.
 We must note that such type of problems could be treated also with the Maximum Principle approach, as done in different contexts, for example,
\cite{LiChenWu2020}, \cite{Yu2012} \cite{WangZhang2013}, \cite{ChenWu2010} \cite{ExarchosTheodorou2018}.

\smallskip

 The novelty of this paper is the  introduction of a finite retirement date,  which leads to two main  mathematical problems. First, conversion of
  the borrowing constraint into a Markovian form is
  %{\mygreen
  considerably more challenging than in the infinite horizon case considered in \cite{BGP}. Second, the infinite dimensional stochastic control problem now features
discontinuous dynamics and time dependent state constraints.  As far as we know,  the treatment of such family of problems  appears to be new in the literature. %\red{\cite{YCW2022,YW2022}, for example, consider finite horizon problems without trading constraints.}
As in \cite{BGP}, the key is the derivation of human capital in explicit form, so that the constraint on total wealth can be properly addressed.  The result is more involved than the one offered in  \cite{BGP}, but reveals novel insights into how the market value of human capital trades off the past contribution of labor income against the shrinking time to retirement. This has considerable bearing for the optimal controls, which are time varying not only in recognition of the residual time to retirement, but also because of the interplay between the past and future contribution  of labor income to human capital and hence total wealth. In line with \cite{VICEIRA}, a hedging demand arises in the allocation to the risky asset because of the market risk channelled by risky wages (e.g., \cite{CV}), but its structure now supports a wider range of empirical predictions than found in the extant literature, as hedging demand is shaped by the joint effect of path dependent wages and the residual time to retirement.

%In addition to considering a fixed retirement date, the present work also introduces the possibility that both idiosyncratic and systematic risk drive labor income. In \cite{BGP}, labor income is fully spanned by the risky assets, whereas here we also consider the case of imperfect instantaneous correlation between labor income and risky assets. As is well known, this makes the market incomplete and hence the valuation of human capital is no longer unique. The general case is notoriously difficult to treat and the extant literature does not seem to offer %{\mygreen
%a way out in our setting; see the discussion in section~\ref{PAR_CORR}. We consider therefore the special case in which the agent's private valuation of human capital relies on the minimal martingale measure. This parametrization is parsimonious and delivers an intuitive structure of the optimal controls, which can be related to the economic intuition developed, for example,
%in \cite{VICEIRA} and \cite{HENDERSON}, among others. In particular, the optimal allocations favor those risky assets that are least correlated with labor income and hence help the agent hedge labor income risk.

\smallskip

The problem considered in this paper is amenable to interesting extensions and ramifications. A first direction is the consideration of tighter borrowing contraints, which even in the case of fully spanned labor income would prevent the agent from perfectly hedging labor income risk; see \cite{EK-JP1998} and references therein.  A second direction is the case in which labor market participation is endogenous, which can be captured by a controlled labor income process (see \cite{MOSTOVYI2017} for a general semimartingale setting) or by making the retirement date $\tau_R$ a stopping time controlled by the agent (as in \cite{DYBVIG_LIU_UNPUBLISHED,DYBVIG_LIU_JET_2010} in a Markovian environment); {see \cite{CHOI2008,BBM2023} for the case in which both continuous labor supply and irreversible retirement are endogeneous.} These extensions are an open problem within our setting and are left for future research.

\smallskip

The paper is organized as follows. In the next section, we introduce the model and  associated expected utility maximization problem. Section~\ref{sec3} then deals with the valuation of human capital and the Markovian representation of the borrowing constraint. This follows from a non trivial analysis, part of which is
relegated to the Appendix. We then rewrite the
path dependent state equation into a Markovian stochastic differential equation (SDE) in infinite dimension. This allows us to reformulate the original portfolio problem in a fully Markovian way.
Section~\ref{SSE:PBREFORMOLATED} is dedicated to the decoupling of the Markovian problem into a two-stage optimal control problem. In particular, we consider in sequence the problems faced by the agent before retirement (finite horizon) and after retirement (infinite horizon) .
In Sections~\ref{Vih_sec} and \ref{Vfh_sec},
we apply the Dynamic Programming (DP) approach to each of the
two problems separately. The infinite horizon problem follows the results of \cite{BGP}; see Section~\ref{Vih_sec}.
The finite horizon problem, which is fully solved in Section \ref{Vfh_sec}, is more challenging, as the associated infinite dimensional HJB equation is of parabolic type and with time dependent
state constraints.  By making an educated guess, however, we can find an explicit solution whose
verification requires some effort. We obtain the optimal controls in explicit feedback form, which is used to discuss
the properties of the optimal strategies.
Finally, Section~\ref{main_result_section}  offers a summary of the main results and their financial implications.
All proofs are put in the appendix which is available online.

\smallskip

\section{Assumptions and problem statement}
\label{Problem formulation}
\smallskip

On a filtered probability space $(\Omega, \mathcal F, \mathbb F, \mathbb P)$,  consider the price of a riskless bond $S_0$ and  an  adapted, vector valued process $S$ representing the price evolution of $n$ risky assets, $S=(S_1,\ldots,S_n)^\top$, with dynamics given by

\begin{eqnarray}\label{DYNAMIC_MARKET}
\left\{\begin{array}{ll}
{\rm d}S_0(s)= S_0(s) r  {\rm d}s, & S_0(0)=1;
\\
{\rm d}S(s) =\text{diag}(S(s)) \left(\mu {\rm d}s + \sigma {\rm d}Z(s)\right), &  S(0)\in {\mathbb R}^n_{+},
%\\ S_0(0)=1\\ S(0)\in {\mathbb R}^n_{+},
\end{array} \right.
\vspace{-0.3truecm}
\end{eqnarray}
where we assume the following.

\smallskip

\begin{hypo}\label{hp:S}
\begin{itemize}
\item[] $\mbox{ }$ \vspace{0.1cm}

\item[(i)]
$Z$ is an $n$-dimensional (standard) Brownian motion. The filtration $\mathbb F:=(\mathcal F_t)_{t \ge 0}$ is augmentation of the one generated by $Z$.

\item[(ii)] the parameter vector $\mu$ is $\mathbb R^n$-valued and the matrix $\sigma \in  \mathbb R^{n \times n}$ is invertible.
\end{itemize}
\end{hypo}

%\red{be invertible and satisfy $\sigma\sigma^\top>0$}, with $\mathbb R$ ($\mathbb R_+$) denoting the (non-negative) real numbers.

\smallskip

An agent is endowed with initial wealth $w\ge 0$ and receives wages at labor income rate $y$ until retirement or death, whichever occurs first. The death time  $\tau_{\delta}$  is modeled as an exponential random variable with parameter $\delta >0$,  independent of $Z$. The retirement date is  fixed  and set equal to $\tau_R \in \mathbb{R}_+$.
%\footnote{We do not introduce a retirement date to keep the model simple. See comments in section~\ref{SE:DISCUSSION} for more details.}
  The reference filtration is the minimal enlargement $\mathbb G := \big( \mathcal G_t \big)_{t \ge 0}$ of $\mathbb{F}$ satisfying the usual conditions and making the death indicator process  $D := (\mathbb{I}_{[0,\tau_{\delta}]}(t))_t$ adapted. Equivalently, we are considering the minimal enlargement of $\mathbb F$ making $\tau_{\delta}$ a  stopping time.  Each sigma-field $\mathcal{G}_t$ can be shown to be given by (e.g., \cite[Proposition 1.12]{AKSAMITJEANBLANC17}):

$$
\mathcal G_t:=   \mathcal F_t \vee
\sigma_g\left( D_u: u \leq t \right)
$$
where $\mathcal{A}\vee \mathcal{B}$ indicates the sigma algebra generated by $\mathcal{A,B}$ and   $\sigma_g ( \cdot)$   denotes the sigma-field generated by the random variable in brackets.
By  \cite[Proposition 2.11-(b)]{AKSAMITJEANBLANC17}, we know that for every   ${\mathbb G}$-predictable process $\tilde{A}$  we can find a  process  ${A}$ which is ${\mathbb F}$-predictable and satisfies, $\mathbb P$-almost surely, the followig condition:

 \begin{equation}\label{eq:GFpred}
 A(s,\omega) = \widetilde{ A}(s,\omega) \qquad
 \forall s \in [0,\tau_\delta(\omega)].
\end{equation}
The processes $A$ and $\widetilde{A}$ are therefore indistinguishable up to $\tau_\delta$. We refer to ${A}$ as to
the pre-death version of $\tilde{A}$.  From now on, we will work with  ${\mathbb F}$-predictable,
pre-death versions of all processes considered,
including state variables and controls.

The agent consumes at rate $c(\cdot)\geq 0$ and can invest in the riskless and risky assets.  We denote by the vector $\theta(\cdot)\in \mathbb R^n$ the wealth amounts allocated to the risky assets.
The agent can also purchase life insurance to reach a bequest target $\widetilde B(\tau_\delta)$ at death. In the pre-death world
this simply means that the agent chooses a control $B(\cdot)\geq 0$.
The pre-death insurance premium paid by the agent is $\delta(B(t)-W(t))$.
As in \cite{DYBVIG_LIU_JET_2010}, we interpret a negative face value $B(t)-W(t)<0$ as a term life annuity trading wealth at death for a positive income flow $\delta(W(t)-B(t))$ while living.   The controls $c, \theta$, and $B$  are for the moment assumed to belong to the following set:

\begin{align}
\Pi_0&:= 	
  \Big\{ \mathbb F\mbox{-predictable triplets }
\left(\left(   c(\cdot), B(\cdot) \right), \theta(\cdot)\right)
\mbox{ in }   \notag \\
& L^1_{loc} (\Omega \times \R_+; \mathbb R^2_{+}) 	
\times L^2_{loc}(\Omega \times \R_+ ; \mathbb R^n)\Big\}, \label{DEF_PI0_FIRST_DEFINITION}
\vspace{-0.5truecm}
\end{align}
where $f\in L^p_{loc} (\Omega \times \R_+; U)$ stands for the integrability condition
$\E [\int_0^T |f(t)|^p dt]< +\infty$ being satisfied for all $T>0$.
We denote by $W$ the agent's financial wealth  and by $y$ the   labor income rate process.
%\mcr{
As opposed to standard bilinear SDEs (e.g., \cite{DYBVIG_LIU_JET_2010}), we assume here  the process $y$   to follow  a bilinear SDDE with delay in the drift. For fixed memory time $d>0$,   the delay term weights the realization of labor income within the time window  $[t-d,t]$ by a time independent function $\phi \in  L^2\left(-d,0; \mathbb R\right)$.  The  dynamics of the state variable pair $(W,y)$ is then given by:

\small{\begin{equation}
\label{DYNAMICS_WEALTH_LABOR_INCOME}
\begin{cases}
 {\rm d}W(s) =  \big[W(s) r + \theta(s)^\top (\mu-r\mathbf{1})
+ (1-R(s))y(s) - c(s)\\
  \ \qquad \qquad -\delta\left(B(s)-W(s)\right)\big] {\rm d}s
 + \theta(s)^\top \sigma
{\rm d}Z(s), \\[1mm]
 {\rm d}y(s)  =  \left[ y(s) \mu_y
+\int_{-d}^0 \phi(s) y(\zeta+s) {\rm d}\zeta  \right]{\rm d}s
+ y(s)\sigma_y^\top    {\rm d}Z(s),\\[1mm]
W(0) = w,\qquad y(0)=  x_0, \quad y(\zeta) = x_1(\zeta)
\mbox{ for $\zeta \in  [-d,0)$}.
\end{cases}
%\vspace{-0.1truecm}
\end{equation}}
with $s\ge0$, $\mu_y,w,x_0 \in  \mathbb R$ and where the triplets $(c(\cdot),B(\cdot),\theta(\cdot))$ are as in
\eqref{DEF_PI0_FIRST_DEFINITION},
$\sigma_y \in  \mathbb R^n$,
$\mathbf 1 = (1,\dots, 1)^\top$ denotes the unitary vector in $\mathbb R^n$
and
$x_1(\cdot)$ lives in $L^2\left(-d,0; \mathbb R\right)$.
%The retirement indicator term (which is the main novelty of our setting) is given by the function  $R(s):= %\pmb{1}
%\mathbb I_{\{\tau_R \le s \}}$. The case $R(s) = 0$~a.e. ($\tau_R = + \infty$) was considered in \cite{BGP}, and corresponds to the case in which the agent works throughout her lifetime.
The function  $R(s):=   \mathbb I_{\{\tau_R \le s \}}$ represents a retirement indicator,
%{\mygreen
where $\tau_R\in \mathbb R_{++}$ represents a deterministic retirement date. This realistic feature is absent in \cite{BGP}, which only considers the case without retirement, i.e., $\tau_R = + \infty$ and $R(s) = 0$~a.e..
% In Section~\ref{PAR_CORR}, Remarks \ref{rem:beta-general-corr} and \ref{VICEIRA}, we will consider an extension of the setup to the case in which labor income is imperfectly correlated with the risky assets. Equivalently, labor income will be driven both by market risk and idiosyncratic risks.
Existence and uniqueness of a strong solution %(with $\mathbb P$-a.s. continuous paths)
to the SDDE for $y$ is ensured by Theorem I.1 and Remark~I.3(iv) %page 144
in \cite{MOHAMMED_BOOK_96}.
%a.s. continuous trajectories for $y$ see page 40 in \cite{MOHAMMED_BOOK_96} Theorem III.4
%We postpone to Section \ref{Section_Infinite dimensional state} the discussion
%on the existence and uniqueness of a solution of the equation for the labor income $y$.
Existence and uniqueness of a strong solution to the SDE for $W$ then follows, for example, from
\cite[Chapter 5.6.C]{KARATZSAS_SHREVE_91}.
As is clear from dynamics \eqref{DYNAMICS_WEALTH_LABOR_INCOME}, in our baseline model labor income is perfectly instantaneously correlated with the risky assets. This is the benchmark that will be used to develop the solution of the portfolio choice problem.

We study the problem of maximizing the agent's expected utility of lifetime consumption and bequest:

$$
\mathbb E \left[\int_{0}^{\tau_\delta} e^{-\rho s }
\left( \frac{(K^{R(s)}\widetilde c(s))^{1-\gamma}}{1-\gamma}
+  \frac{\big(k \widetilde B(\tau_\delta)\big)^{1-\gamma}}{1-\gamma}\right) \,{\rm d}s
\right],
$$
which can be reduced to the following one after exploiting the fact that the death time is exponentially distributed:
{\small{ \begin{eqnarray}\label{OBJECTIVE_FUNCTION}
\mathbb E \left[\int_{0}^{+\infty} e^{-(\rho+ \delta) s }
\left( \frac{(K^{R(s)}c(s))^{1-\gamma}}{1-\gamma}
+ \delta \frac{\big(k B(s)\big)^{1-\gamma}}{1-\gamma}\right) \,{\rm d}s
\right].
\end{eqnarray}}
Optimization is carried out over all triplets $\left(c,\theta,B\right)$ taken as in
\eqref{DEF_PI0_FIRST_DEFINITION} and satisfying a suitable state constraint introduced in
\eqref{NO_BORROWING_WITHOUT_REPAYMENT_CONDITIONLA_MEAN} further below.
%\mcr{
In the above, the parameter $K>1$ allows the utility from consumption to
differ before and after time $\tau_R$, whereas $k>0$ measures
the intensity of preference for leaving a bequest. The parameter $\gamma$
is the relative risk aversion coefficient, which is assumed to satisfy $\gamma \in (0,1) \cup (1, +\infty)$. Finally, the
parameter  $\rho >0$ denotes the agent's subjective discount rate.
%\mcg{Disintegrating by the values of $\tau_\delta$ we get the final form of the objective function:}
%As discussed in \cite{BGP}, we can work with pre-death processes
%and consider the following objective function:
%
%Let us now introduce the state constraint that is natural in our context.

In the reference financial market  \eqref{DYNAMIC_MARKET},  the pre-death state-price density
of the agent
%\mcg{
is unique and is given by:

\begin{equation}\label{DYN_STATE_PRICE_DENSITY}
\left\{\begin{array}{ll}
{\rm d} \xi (s)& = - \xi(s)(r +\delta) {\rm d}s
-\xi(s) \kappa^\top {\rm d}Z(s),\\
\xi(0)&=1.
\end{array}\right.
\end{equation}
where $\kappa$ is the market price of risk and is defined as follows
(e.g., \cite{KARATZSAS_SHREVE_91}):
\begin{equation}\label{DEF_KAPPA}
\kappa:= (\sigma)^{-1} (\mu- r \mathbf 1).
\end{equation}
%We then require the agent to satisfy, at every time
%$s \ge 0$, the following constraint:
Denote  now by   $H$  the
%\mcr{
current market
%capitalized
%\mcg{the present}
value of future wages, or human capital:
\begin{equation}
\label{HC}
H(s):=  \xi^{-1}(s)\mathbb E\Big( \int_s^{+\infty} \xi(u) (1-R(u)) y(u)
 \, {\rm d}u \Big \vert  \mathcal F_s\Big).
\end{equation}
%We refer to $H$ as to the human capital process.
The agent's budget constraint is then assumed to satisfy:
\vspace{-0.3truecm}
\begin{equation}\label{NO_BORROWING_WITHOUT_REPAYMENT_CONDITIONLA_MEAN}
W(s) +   H(s)  \geq 0, \ \ \ \forall  s\in [0,\infty).
\vspace{-0.3truecm}
\end{equation}
The above means that human capital, in addition to financial wealth, can be pledged as collateral
%\mcr{
to support investment and consumption.  %\mcg{when investing, consuming or selecting the bequest.
Note that the agent cannot default on his/her debt obligations upon death, as
$B$ is assumed to be nonnegative.
Constraint \eqref{NO_BORROWING_WITHOUT_REPAYMENT_CONDITIONLA_MEAN}
was introduced as a no-borrowing-without-repayment condition in \cite{DYBVIG_LIU_JET_2010}. Its impact on portfolio choice was explored numerically  in the context of mean-reverting labor income dynamics chosen as a proxy for sticky wages.
Let us denote by $W^{w,x}(s; c,B,\theta)$ and $y^{x}(s)$ the solutions
of system \eqref{DYNAMICS_WEALTH_LABOR_INCOME} at time $s\geq 0$, and by $H^{x}(s)$ the corresponding human capital defined in \eqref{HC},  where
we emphasize the dependence of these quantities
on the initial conditions $(w, x)$ and strategies
$(c,B,\theta)$, where relevant.
We can then define the set of admissible controls, as follows:
\begin{align}
\Pi\left(w,x\right) &:=
\Big\{%\mathbb F-\mbox{predictable }
\left(c(\cdot), B(\cdot), \theta(\cdot)\right)
\in
\Pi_0
\mbox{ \ such that:}\notag \\
&  W^{w,x}\left(s; c,B,\theta\right) +  H^{x}(s)\geq 0
 \,\quad \forall \ s\geq 0 \Big\}.
\label{DEF_PI_FIRST_DEFINITION}
\end{align}
Our problem is then to maximize the functional given
in \eqref{OBJECTIVE_FUNCTION}
over all controls in $\Pi\left(w,x\right)$.
Let us now
%{\mygreen
introduce the effective discount rate for labor income (e.g., \cite{DYBVIG_LIU_JET_2010}), which is assumed to be positive:
\vspace{-0.4truecm}
\begin{equation}\label{eq:defbeta}
\beta:= r+\delta - \mu_y+  \sigma_y^\top \kappa> 0.
\vspace{-0.3truecm}
\end{equation}
%{\mygreen
Consider also the following  assumption, which
is the analog of Merton's condition for the maximization problem to be well-posed and will be recalled explicitly
whenever needed:
%Some results will hold without using it,
%hence we will mention explicitly when it is used.
\vspace{-0.3truecm}
\begin{hypo}
\label{hp:betarho}
%{\mygreen
The model parameters satisfy the inequality
$\rho + \delta -(1-\gamma)
(r + \delta +\frac{\kappa^\top \kappa}{2\gamma }) >0$.
%\begin{itemize}
%  \item [(i)]
%$\beta:= r+\delta - \mu_y+  \sigma_y^\top \kappa> 0$.
%  \item [(ii)]
%$\rho + \delta -(1-\gamma)
%(r + \delta +\frac{\kappa^\top \kappa}{2\gamma }) >0$.
%\end{itemize}
\end{hypo}

\vspace{-0.3truecm}

\begin{rem}
%The parameter $\beta$
%(assumed to be strictly positive)
%represents the effective discount rate for labor income (see e.g., [15]).
%\mygreen{
In paper \cite{BGP}, a transversality condition is required on account of the infinite horizon problem; this imposes an additional restriction on $\beta$, see
  \cite[Hypothesis 2.4-(i)]{BGP} and Remark~\ref{rm:lesshpforconstraint} further below.
Hypothesis~\ref{hp:betarho}, on the other hand, is required here to ensure that the value function is finite and coincides with \cite[Hypothesis 2.4-(ii)]{BGP}, as it depends only on the target functional and the wealth equations;
see Remark \ref{rm:newsec5} for further discussion.
%Without this assumption, in simple deterministic cases, one can prove that the value function is not finite (see, e.g., \cite{FreniGozziSalvadori06}).
\end{rem}

\vspace{-0.3truecm}

Our problem can be summarized as follows:

\vspace{-0.3truecm}

\begin{prob}
\label{p0}
%\mcr{
Given initial wealth $w$, initial income and income history pair $x=(x_0,x_1)$, and
%deterministic time to retirement $\tau_R$, PREFERISCO DIRE RETIREMENT
retirement date $\tau_R$,
with associated indicator function
$R(s):= \mathbb I_{\{\tau_R \le s \}}$,
choose consumption, bequest, and investment triplets
$\left(c(\cdot),B(\cdot),\theta(\cdot)\right)$ so as
to maximize expected utility
\eqref{OBJECTIVE_FUNCTION}
under state equation \eqref{DYNAMICS_WEALTH_LABOR_INCOME}
and subject to the borrowing constraint
\eqref{NO_BORROWING_WITHOUT_REPAYMENT_CONDITIONLA_MEAN}.
\end{prob}

\vspace{-0.3truecm}

\section{Rewriting the state equation and the constraint}\label{sec3}

\vspace{-0.3truecm}

In section~\ref{sec:infdimframework}, we briefly recall
how to rewrite the labor income equation as a Markovian SDE in the
Delfour-Mitter space $M_2$, which is an infinite dimensional Hilbert space.
In Subsection \ref{Rephrasing the no-borrowing constraint},
we derive the explicit expression for the agent's human capital
and rewrite the constraint in $M_2$.

\subsection{Reformulation of the SDDE for %the
labor income.}
\label{sec:infdimframework}

{We follow the main ideas of \cite[Subsection 3.1]{BGP}, but consider
the generic initial time $t\ge 0$ instead of time zero.
%We present here the main ideas for the reader's convenience.
The state equation for the labor income $y$ is an SDDE, which means that $y$ is not Markovian and the dynamic programming principle
does not apply in its standard formulation.}
As usual, it is convenient to reformulate the problem in an
infinite dimensional Hilbert space, which takes into account both the present and past values of the states, thus
ensuring  the Markov property of the states
(see, for example, \cite{VINTER} and
\cite{CHOJNOWSKA-MICHALIK_1978}, or
\cite[Section 0.2]{DAPRATO_ZABCZYK_RED_BOOK} and
\cite[Section 2.6.8]{FABBRI_GOZZI_SWIECH_BOOK}).

To be precise, we introduce the Delfour-Mitter Hilbert space
$M_2$ (see e.g. \cite[Part II, Chapter 4]{BENSOUSSAN_DAPRATO_DELFOUR_MITTER}):

\begin{equation*}
M_2 := \mathbb R \times L^2\big( -d, 0; \mathbb R \big),
\end{equation*}
with inner product, for $x=(x_0, x_1), y=(y_0, y_1) \in M_2$, defined as
$\langle x,y \rangle_{M_2}:= x_0 y_0 + \langle x_1, y_1 \rangle_{L^2}$,
where
$$
\langle x_1, y_1 \rangle_{L^2}:=\int_{-d}^0 x_1(\zeta) y_1(\zeta) d\zeta.
$$
For ease of notation, we will drop below the subscript $L^2
%\mcg{
= L^2\big( -d, 0; \mathbb R \big)$
from the inner product of such space, writing simply
$\langle x_1, y_1\rangle$.
To embed the state $y$ of the original problem
in the space $M_2$ we introduce the linear operators
$A$ (unbounded) and $C$ (bounded).
Setting
\vspace{-0.3truecm}
\begin{equation*}
\mathcal D(A) :=\left\{(x_0,x_1) \in M_2: x_1(\cdot) \in
W^{1,2}\left( [-d, 0]; \mathbb R \right), x_0 = x_1(0)\right\},
\end{equation*}
the operator $A: \mathcal D (A) \subset M_2 \rightarrow M_2$ is defined as
\begin{equation}
\label{DEF_A}
A(x_0,x_1) := \left(\mu_y x_0 +\langle \phi,x_1 \rangle,
 x_1^{\prime}  \right),
\end{equation}
with $\mu_y,\phi$ appearing in equation (\ref{DYNAMICS_WEALTH_LABOR_INCOME}).
The operator $C:M_2 \rightarrow   \mathcal{L}(\R^n,\R)$
is bounded and defined as
\begin{equation*}
C(x_0,x_1):= x_0  \sigma_y^\top
\end{equation*}
where $\sigma_y $ shows up in (\ref{DYNAMICS_WEALTH_LABOR_INCOME}).
Proposition A.27 in \cite{DAPRATO_ZABCZYK_RED_BOOK} shows that
$A$ generates a strongly continuous semigroup in $M_2$.
Consider, for $t\ge 0$, $x\in M_2$,
the following stochastic differential equation in $M_2$:
\begin{equation}
\begin{cases}
\label{INFINITE_DIMENSIONAL_STATE_EQUATION}
%\left\{\begin{array}{ll}
d X(s) = A X(s) ds+ C X(s)dZ(s)\\
  X(t) = x \in M_2.
\end{cases}
\end{equation}
The equivalence between the equation for $y$ in \eqref{DYNAMICS_WEALTH_LABOR_INCOME} and the above \eqref{INFINITE_DIMENSIONAL_STATE_EQUATION} is formalized in the following Proposition.

\begin{prop}
\label{equiv_stoch}
For every $t\ge 0$ and $x=(x_0,x_1)\in M_2$ the
%\mcg{
infinite dimensional equation \eqref{INFINITE_DIMENSIONAL_STATE_EQUATION} has a unique mild (and weak) solution with almost surely continuous trajectories,
denoted by $X^x(\cdot;t)=(X_0^x(\cdot;t),X_1^x(\cdot;t))$.
Moreover the equation for $y$ in (\ref{DYNAMICS_WEALTH_LABOR_INCOME}), possibly with initial time $t \ge 0$,
has a
%\mcg{
unique mild solution, in a strong probabilistic sense, for every $x\in M_2$,
which we denote by $y^x(\cdot;t)$.
Finally, for every $s\ge 0$ and $\zeta \in [-d,0)$,
\vspace{-0.3truecm}
\begin{equation} \label{eq:Yy} \left(X_0^x(s;t),X_1^x(s;t)(\zeta)\right)=
\left(y^x(s;t),y^x(s+\zeta;t)\right).
\vspace{-0.3truecm}
\end{equation}
\end{prop}

We will denote the mild solution of (\ref{INFINITE_DIMENSIONAL_STATE_EQUATION}) with $X^x(\cdot)$
to underline its dependence on the initial condition. When no confusion is possible we may simply write $X(\cdot)$.
Note that $X$ now
% differently from $y$
 is Markovian.

We conclude this Section by a standard result about the adjoint operator of $A$ which will be used to derive the explicit solution of the HJB equation
in Section \ref{Vfh_sec}.

\begin{prop} \label{adj}
The adjoint operator of $(A, \mathcal{D}(A))$
is $A^*: \mathcal D(A^*)\subset M_2 \longrightarrow M_2$
and is given by:
\vspace{-0.3truecm}
\begin{equation*} \mathcal D(A^*):=\{(z_0,z_1) \in M_2: z_1 (\cdot) \in  W^{1,2}\big(-d, 0; \mathbb R \big), z_1(-d) = 0\},
\vspace{-0.3truecm}
\end{equation*}
\vspace{-0.3truecm}
\begin{equation} \label{EQ_ADJOINT_OPERATOR_A*} A^*(z_0,z_1):= \big(\mu_y z_0 +z_1(0), - z_1' + z_0 \phi \big).
\vspace{-0.4truecm}
\end{equation} \end{prop}

\noindent \subsection{A formula for human capital}
\label{Rephrasing the no-borrowing constraint}

After the preliminary Lemma \ref{LEMMA_EXISTENCE_g_h},
in Proposition \ref{PROPOSITION_EVALUATION_DELAYED_LABOR_INCOME} we derive
an explicit expression for the agent's human capital in the infinite
dimensional framework introduced in the previous section. The expression  extends the result
provided by \cite{BGPZ} to a finite horizon setting.
We then use it  in Corollary \ref{cor3.3}
%\mcg{
to  suitably rewrite the
constraint
(\ref{NO_BORROWING_WITHOUT_REPAYMENT_CONDITIONLA_MEAN}).
%To this end,
% Section introduces completely new material relative to
%to the infinite horizon case treated in \cite{BGP} and \cite{BGPZ}.
% \mcg{
%we prove novel results with respect to the infinite horizon case  in \cite{BGP} and \cite{BGPZ}.
\vspace{-0.3truecm}

\begin{lemma}
\label{LEMMA_EXISTENCE_g_h}
Let $\phi\in L^2(-d,0;\R)$ and $d,\tau_R >0$, $\beta \in \R$. Consider the system
\hspace*{0.2cm}
\vspace{-0.3truecm}
\begin{equation}\label{DEF_g_h}
\begin{cases}
g(t)=   \mathbb I_{\{t < \tau_R\}}\int_t^{\tau_R } e^{-\beta(\tau-t)}
 \big(h(\tau,0)+1 \big)\, {\rm d} \tau,
 \\
h(t,\zeta)=  \mathbb I_{\{t < \tau_R\}}* \\
\hfill*\int^\zeta_{(\zeta+t-\tau_R)\vee (-d)}   e^{- (r+ \delta) (\zeta-\tau)}
  g(t+\zeta-\tau)\phi(\tau)\,{\rm d}\tau,\\
 \mbox{with }   (t,\zeta)
  \in [0, \infty) \times [-d,0].
 \end{cases}
 \end{equation}
We have the following results:

\begin{itemize}
  \item[(i)]
The system \eqref{DEF_g_h} admits a unique solution, i.e., a pair $(g,h)$, with $g \in C(\R_+;\R)$ and $h\in C(\R_+\times [-d,0];\R)$, satisfying \eqref{DEF_g_h}.
Moreover $g$ and $h$ depend continuously on $\phi\in L^2(-d,0;\R)$. If $\phi$ is a.e. positive, then also $(g,h)$ is positive.

\item[(ii)]
Assume now that $\phi\in C^1([-d,0];\R)$ with $\phi (-d)=0$. Then the solution $(g,h)$ of  \eqref{DEF_g_h} belongs to $C^1([0,\tau_R];\R)\times C^1([0,\tau_R]\times [-d,0];\R)$ and satisfies the system
\begin{equation}
\label{CAUCHY_PROBLEM_g_h}
\begin{cases}
g^{\prime}(t) -\beta g(t)  +h(t,0) +1 =0, \hfill  t \in [0,\tau_R)
\vspace{0.1cm}\\
-(r+ \delta)h(t,\zeta)
+  \frac{ \partial h(t,\zeta)}{\partial t}
-\frac{ \partial h(t,\zeta)}{\partial \zeta}+ g(t) \phi(\zeta)  =0,
\\   \hfill t \in (0,\tau_R), \zeta \in (-d,0)
\vspace{0.1cm}\\
g(t)=0, \hfill  t \geq \tau_R
\vspace{0.1cm}\\
h({t},\zeta)=0, \hfill t\geq \tau_R, \, \zeta \in [-d,0]
\vspace{0.1cm}\\
h(t, -d)=0, \hfill  t \in [0,+\infty).
\end{cases}
\end{equation}

\item[(iii)]
If $\phi\in L^2(-d,0;\R)$ we have $g\in C^1([0\tau_R];\R)$
and the map
\vspace{-0.2truecm}
$$
[0,\tau_R] \to L^2(-d,0;\R), \qquad t \mapsto h(t,\cdot)
\vspace{-0.2truecm}
$$
belongs to $ C^1([0\tau_R];L^2(-d,0;\R))$.
From now on we will write $h(t)$ for $h(t,\cdot)$.

\item[(iv)]
It holds that $\big(g(t),h(t)\big) \in \mathcal D(A^*)$
for any $t \geq0$. Moreover
$(g,h) \in L^{\infty}(0,\tau_R;\mathcal{D}(A^*))$, that is
\vspace{-0.2truecm}
\begin{equation*}
\sup_{t \in [0, \tau_R)}\|(g(t),h(t)) \|_{\mathcal{D}(A^*)} < \infty.
\vspace{-0.2truecm}
\end{equation*}
Finally, the couple $(g,h)$ is a mild solution
(see e.g. \cite[Definition 1.119]{FABBRI_GOZZI_SWIECH_BOOK}
for the definition) of the Cauchy problem in $L^2(-d,0;\R)$:
\vspace{-0.2truecm}
\begin{equation}\label{CAUCHY_PROBLEM_g_hnew}
  \left(g(t),h(t)\right)'
  = \left(- A^*(g(t),h(t))
  +\left((\beta + \mu_y)g(t) -1, (r+\delta)h(t)\right)\right),
\vspace{-0.2truecm}
\end{equation}
($t\in[0,\tau_R)$) with zero final condition at $t=\tau_R$.
\vspace{-0.2truecm}
\item[(v)]
The function $r \mapsto g'(r)$ has a discontinuity in $r=\tau_R$.
\end{itemize}
\end{lemma}
%\begin{pf}
%We postpone the long and technical proof of this result to the Appendix.
%\end{pf}

\vspace{-0.3truecm}

\begin{prop}
\label{PROPOSITION_EVALUATION_DELAYED_LABOR_INCOME}
Let $x \in M_2$ and let $y$ be the corresponding solution of the second equation of \eqref{DYNAMICS_WEALTH_LABOR_INCOME}.
%Assume that  Hypothesis \ref{hp:betarho}-(i) hold true.
The market value of Human Capital defined in \eqref{HC} admits the following representation: for $s\ge 0$, $\mathbb P$-a.s.

\begin{equation}
\label{1}
H(s)=g(s)y(s)+\int_{-d}^0h(s,\zeta)y(s+\zeta)\, {\rm d}\zeta,
%\quad s \in [0, \infty), \, \mathbb P-a.s.,
\end{equation}
where the function 
$$
(g,h)\in C([0, +\infty);\R_+)\times
C([0, +\infty)\times [-d,0];\R_+)
$$
is the unique solution
of \eqref{DEF_g_h}.
Let now $X$ be the solution of equation
\eqref{INFINITE_DIMENSIONAL_STATE_EQUATION} with initial datum $x$.
Then we have for all $s \in [0, \infty), \quad \mathbb P-a.s.$,
\small{\begin{align}
\label{1bis}
H(s)=\<(g(s),h(s)),X(s) \>_{M_2}=g(s)X_0(s)+ \langle h(s),X_1(s)\rangle.
\end{align}
}
\end{prop}

As a direct consequence of Proposition \ref{PROPOSITION_EVALUATION_DELAYED_LABOR_INCOME} we rewrite the constraint given in \eqref{NO_BORROWING_WITHOUT_REPAYMENT_CONDITIONLA_MEAN} in a more convenient way.

\begin{cor}
\label{cor3.3}
%Let Hypothesis \ref{hp:betarho} hold.
For any $s \in[0,+\infty)$,
the constraint
(\ref{NO_BORROWING_WITHOUT_REPAYMENT_CONDITIONLA_MEAN})
can be equivalently rewritten as
\begin{eqnarray}\label{CONSTRAINT_FIN_RET}
W(s)+  g(s) X_0(s) +  \langle h(s) , X_1(s) \rangle \geq 0.
\end{eqnarray}
\end{cor}

\begin{rem}\label{rm:lesshpforconstraint}
The above Corollary \ref{CONSTRAINT_FIN_RET} is the finite retirement version of \cite[Theorem 2.1]{BGPZ} (see also
\cite[Proposition 3.3]{BGP}).
As we have finite retirement, here
we do not need to care about the long run behavior of labor income.
This means that we do not need additional restrictions
on $\beta$ and $\phi$ (e.g., \cite[Hypothesis 2.4-(i)]{BGP}).
\\
On the other hand,  the finite retirement time makes the rewriting of the constraint much more complicated, as the functions $g$ and $h$ now depend on time; this is the reason why we need to prove the long technical Lemma \ref{LEMMA_EXISTENCE_g_h}.
\end{rem}

\vspace{-0.3truecm}

We conclude this subsection with a remark on the financial meaning of the key functions $g$ and $h$.

\vspace{-0.3truecm}

\begin{rem}\label{rm:ghmeaning}
To better understand the role of the functions   $g$ and $h$ appearing in \eqref{DEF_g_h}, let us first rewrite $g$ by distinguishing two different components, $g_1$ and $g_2$:
\vspace{-0.2truecm}
\begin{align}\label{g12}
&g(t)=\mathbb I_{\{t<\tau_R\}}\int_t^{\tau_R } e^{-\beta(\tau-t)} \, {\rm d} \tau + \\
& + \mathbb I_{\{t<\tau_R\}} \int_t^{\tau_R } e^{-\beta(\tau-t)}
h(\tau,0) \, {\rm d} \tau= g_1(t)+g_2(t). \notag
\vspace{-0.2truecm}
\end{align}
The function $g$ is essentially an annuity factor capturing the market value of a stream of unitary wages to be received until retirement.
%In our model, the annuity factor is characterized by two components.
Ignoring the delay term in income dynamics, the cumulative discounted value of each unit of labor income to be received until retirement is given by
\vspace{-0.2truecm}
\begin{equation}
\label{g1}
g_1(t)=\mathbb I_{\{t<\tau_R\}} \int^{\{\tau_R\}}_t e^{-\beta (\tau-t)} \, {\rm d}\tau = \frac{1-e^{-\beta (\tau_R-t)^+}}{\beta},
\vspace{-0.2truecm}
\end{equation}
which is nothing else than the usual annuity term appearing in models without delay (e.g., \cite{DYBVIG_LIU_JET_2010}).
%{\mygreen
In our context, however, each unit of labor income received also affects the wages to be received  in the future. %This is exactly what the second component, $g_2$, is capturing. To understand the mechanics of this channel (???), we consider the following simple example.
%\mcg{
Such effect is precisely captured by the  second component $g_2$, as illustrated in the next simple example. \\
Assume that  $\phi(-1)>0$ and
is zero otherwise on $[-d,0]$, with ${ d>1}$, so that wage realizations influence labor income with a delay of one time unit.
Their  cumulative effect is then given  by
the following equation:
\vspace{-0.2truecm}
\begin{equation}
\label{g2}
g_2(t)=\int_t^{t\vee (\tau_R-1)}
e^{-\beta (\tau-t)} \left(e^{-(r+\delta)}\phi(-1) g(\tau+1)\right)\,{\rm d}\tau .
\end{equation}
The integral above provides the cumulative market value of the delayed contribution of each unit of income received before retirement. At each time $\tau\in  [t,\tau_R-1]$, any unit of income received is weighted by the term $\phi(-1)$ and received with certainty in case of survival with a delay of one unit of time, thus delivering a time-$(\tau+1)$ annuity value of $\phi(-1) g(\tau+1)$. The time-$\tau$ value {to the agent of such a} contribution is therefore $e^{-(r+\delta)}\phi(-1) g(\tau+1)$, and its time-$t$ value is then obtained by applying the standard discount factor. For each time $\tau\in (\tau_R-1,\tau_R]$, the delayed contribution has no time to materialize before retirement and hence  is not accounted for in $g_2$. %\textcolor{blue}{Serve fare un commento anche su $h$?}
\end{rem}

\vspace{-0.3truecm}

%%%%%%%%%%%%%%%%%%%%%%%%%%%%%%%%%%%%%%%%%%%%%%%%%%%%%%%%%%%%%%%%%%%%%%%%%%%%%%%%%%%%%%%%%%%%%%%%%%%%%%%%%%%%

\section{Solving the optimization problem: preliminary work}
\label{SSE:PBREFORMOLATED}

\vspace{-0.3truecm}

In subsection~\ref{SSE:STATEMENT}
we rewrite the problem exposed in section~\ref{Problem formulation} using the results
introduced in section~\ref{sec3} and letting the
initial time $t$ vary.
Then, in subsection~\ref{SSE:STRATEGY} we sketch the idea of how we will solve the problem.
Finally, in section~\ref{SSE:EVOLUTION} we study
the time evolution of admissible paths

\vspace{-0.3truecm}

\subsection{Statement of Problem \ref{p0} rewritten}
\label{SSE:STATEMENT}

For $t\ge 0$ and $(w,x) \in \mathcal{H}:=\mathbb{R} \times M_2$,
we rewrite \eqref{DYNAMICS_WEALTH_LABOR_INCOME} with
initial time $t$ and unknown $(W,X)$ as:
\vspace{-0.4truecm}
\begin{eqnarray}
\label{DYN_W_X_INFINITE_RETIREMENT_II}
\begin{split}
\left\{\begin{array}{ll}
{\rm d}W(s)=\left[ (r+\delta) W(s)+ \theta^\top(s) (\mu-r \mathbf 1)
+ \right.\\
\left. (1-R(s)) X_0(s) - c(s) - \delta B(s) \right]\, {\rm d}s
+  \theta^\top (s) \sigma \, {\rm d}Z(s),\\
 {\rm  d}X(s) = A X(s)\, {\rm d}s + \big(C X(s) \big)_0^\top\,
 {\rm d}Z(s),\\
  W(t)= w, \qquad
X_0(t) = x_0,\quad  \\
X_1(t)(\zeta) = x_1(\zeta)
\mbox{ for $\zeta \in  [-d,0)$.}
 \end{array}\right. \end{split}
\end{eqnarray}
By Proposition \ref{equiv_stoch} the above admits a unique strong
solution $(W,X)$ corresponding to the unique solution
$(W,y)$ of \eqref{DYNAMICS_WEALTH_LABOR_INCOME} when the initial
time is $t$.
We denote such solution at time $s \ge 0$ by
$(W^{w,x}\left(s;t,c,B,\theta\right), X^{x}(s;t))$.
With this notation we emphasize the dependence of the solutions
on the present time $t$, the initial time $0$, the initial conditions
$(w_0,y)\in \mathcal{H}$
and the admissible controls $(c(\cdot),B(\cdot),\theta(\cdot))$. For readability, we will often denote a triplet
of controls $(c(\cdot),B(\cdot),\theta(\cdot))$ as $\pi(\cdot)$, as well as use the shorthand notation
$(W(s; \pi),X(s))$, in which dependence on initial time and conditions
is subsumed.
We take the following set of admissible controls
\vspace{-0.4truecm}
\begin{align}
\hspace*{-0.1cm}& \Pi\left(t,w,x\right)
:= \Big\{ \mathbb F\mbox{-predictable }
\pi(\cdot)=\left(c(\cdot), B(\cdot), \theta(\cdot)\right)
 \text{ in } \notag \\
& L^1_{loc} (\Omega \times [t, +\infty),
\mathbb R^2_{+}) \times L^2_{loc}(\Omega \times [t,+\infty) , \mathbb R^n) \mbox{ and such that }
\notag\\
& W^{w,x}\left(s;t,\pi\right) + g(s)X^x_0(s;t)+
\langle h(s), X^x_1(s;t) \rangle \geq 0\,
\ s \in [t, \infty)\Big\}.
  \label{DEF_PI_SECOND_DEFINITION}
\vspace{-0.4truecm}
\end{align}
\vspace{-0.3truecm}

Thanks to the results proved in Section \ref{sec3}, the above set
coincides with the one of \eqref{DEF_PI_FIRST_DEFINITION} for $t=0$.

\vspace{-0.3truecm}

The objective functional $J \left(t, w, x ; \pi\right)$ is (compare with the
one of \eqref{OBJECTIVE_FUNCTION})
\vspace{-0.4truecm}
\begin{equation}
\label{DEF_J_FIN_RET}
\mathbb E \left[\int_{t}^{+\infty} e^{-(\rho+ \delta) s }
\left( \frac{(K^{R(s)}c(s))^{1-\gamma}}{1-\gamma}\\
+ \delta \frac{\big(k B(s)\big)^{1-\gamma}}{1-\gamma}\right) {\rm d}s
\right].
\vspace{-0.1truecm}
\end{equation}
For simplicity, here we write $J$ as a function of $\pi$,
although it only depends on the control triplet through
$c$ and $B$,  and not through $\theta$.
Notice that the functional $J$ may take the value $- \infty$ when
$\gamma >1$ and both $c(\cdot)$ and $B(\cdot)$ are identically zero.
On the other hand, it cannot take the value $+ \infty$ when $\gamma \in (0,1)$,
when Hypothesis \ref{hp:betarho} holds.
Our target Problem~\eqref{p0} is then reformulated as follows:

\vspace{-0.3truecm}

\begin{prob}
\label{pb1}
For fixed initial state $(w,x) \in \mathcal{H}$
such that $w+ g(0)x_0+ \langle h(0, \cdot),x_1\rangle \ge 0$,
find $\overline\pi \in \Pi(0,w,x)$ such that
\vspace{-0.4truecm}
\begin{equation*}
J(0,w,x; \overline \pi)= \sup_{\pi \in \Pi(0,w,x)}J(0,w,x; \pi).
\end{equation*}
%\vspace{-0.2truecm}
\end{prob}

\subsection{Solving Problem \ref{pb1}: solution strategy}
\label{SSE:STRATEGY}

To solve Problem \ref{pb1}, we use a two-stages optimal control technique
solving Problems \ref{pb2} and \ref{pb3} below, in that order.
Before stating them, let us fix some notation allowing us to rewrite
the total wealth of the agent and the constraint condition in terms
of the function $\Gamma$ defined as follows.
\begin{defn}
\label{notation}
We define the map $\Gamma:\R_+  \times \Hh\to\R$ as
\vspace{-0.3truecm}
\begin{equation}
\label{Gamma}
\Gamma(t,w,x):=w+g(t)x_0+ \langle h(t),x_1\rangle.
\vspace{-0.3truecm}
\end{equation}
It is also convenient to introduce the following subsets
 of $\R_+\times \mathcal{H}$: for $0 \le s< r \le \infty$: %we define
\vspace{-0.3truecm}
\begin{equation}
\label{H+}
\mathcal{H}_+^{s,r}:=\{(t,w,x) \in [s,r] \times \mathcal{H} \ : \ \Gamma(t,w,x )\ge 0\}
\vspace{-0.3truecm}
\end{equation}
and
\vspace{-0.3truecm}
\begin{equation*}
\mathcal{H}_{++}^{s,r}:=\{(t,w,x) \in
[s,r] \times \mathcal{H} \ : \ \Gamma(t,w,x )> 0\}.
\end{equation*}
In the case $0 \le s<  \infty$, by $\mathcal{H}_{+}^{s,\infty}$
we mean the subset $\{(t,w,x) \in [s,\infty)
\times \mathcal{H} \ : \ \Gamma(t,w,x )\ge 0\}$.
The case $\mathcal{H}_{++}^{s,\infty}$ is analogous.
%DECIDERE SE LASCIARE SOLO $\mathcal{H}^{0, \infty}_{++}$
\end{defn}

We now define the value function
$V:{\mathcal{H}^{0, \infty}_{+}}\rightarrow \overline{\mathbb{R}}$
as
\begin{footnote}{Note that at this point we allow $V$ to take the
values $\pm \infty$.}\end{footnote}
\vspace{-0.3truecm}
\begin{equation}
\label{V_t}
V(t,w,x):= \sup_{\pi \in \Pi(t,w,x)}J(t,w,x;\pi).
\vspace{-0.3truecm}
\end{equation}
We note that at time $\tau_R$ the dynamics of our problem changes.
In particular, when $t\ge \tau_R$, $(g(t),h(t))\equiv 0$, hence
$\Pi(t,w,x)$ does not depend on $x$:
in this case we will denote it simply by $\Pi(t,w)$.
If we assume that
the Dynamic Programming Principle (DPP) holds
(see e.g. \cite[Section 2.3]{FABBRI_GOZZI_SWIECH_BOOK})
we can rewrite the value function $V$ as follows.

\begin{itemize}
\item For any $t \in [\tau_R, \infty)$, $V(t,w,x) = V(t,w)=$ and the common value is
\vspace{-0.3truecm}
\begin{equation}
\label{split0}
\sup_{\pi \in \Pi(t,w)}
 \mathbb E \left[\int_t^{\infty} e^{-(\rho+ \delta) s }
\left( \frac{(Kc(s))^{1-\gamma}}{1-\gamma}
+  \delta \frac{\big(k B(s)\big)^{1-\gamma}}{1-\gamma}\right)
\, {\rm d}s \right]. \notag
\vspace{-0.3truecm}
\end{equation}
\item For any $t \in [0, \tau_R]$, $ V(t,w,x) $ equals:
\vspace{-0.3truecm}
 \begin{multline}
 \label{split}
 \sup_{\pi \in   \Pi(t,w,x)}  \mathbb E
  \left[\int_t^{\tau_R} e^{-(\rho+ \delta) s }
\left( \frac{c(s)^{1-\gamma}}{1-\gamma}\right.
+ \left.\delta \frac{\big(k B(s)\big)^{1-\gamma}}{1-\gamma}\right) \,{\rm d}s \right.
\\
\left.
+ V\big(\tau_R,W^{w,x}(\tau_R;t,\pi ),X^{x}(\tau_R;t)\big)
\right].
\vspace{-0.3truecm}
 \end{multline}
\end{itemize}

\vspace{-0.3truecm}

Note that we cannot say, at this stage, if DPP holds,
as we do not even know whether $V$ is finite.
However, we will overcome this difficulty by proceeding as follows.
\begin{itemize}
  \item{\bf Step (I)}
  First, we solve
  problem \eqref{split0}
  finding its value function $V^{ih}$ (which must be equal to $V$ by construction) and the optimal strategies
  when $t\ge \tau_R$. We call this ``Problem~3'' and solve it
  in next section (Section \ref{Vih_sec}).
\vspace{-0.3truecm}
  \item{\bf Step (II)}
  Second, we consider problem \eqref{split}
      with the final payoff given by the solution found at Step (I).
  We find again its value function $V^{fh}$ and
  its optimal strategies. We call this `Problem~4'' and solve it
  in Section \ref{Vfh_sec}.
\vspace{-0.3truecm}
  \item{\bf Step (III)}
   Observe that $V^{fh}=V$ and then conclude
  at the end of Section \ref{Vfh_sec}.
\end{itemize}

\vspace{-0.3truecm}

\begin{rem}
The strategy we follow to solve the optimization Problem \ref{pb1} relies
on finding an explicit solution of the HJB and prove
that this is equal to the value function.
We employ the two-stages optimal control technique and solve,
in this order, Problems \ref{pb2} and \ref{pb3}
since the HJB equations associated to these problems are smooth in time.
Smoothness in time is instead lost if we consider the HJB equation
associated to Problem \ref{pb1}. Intuitively this is evident
since in $\tau_R$ the dynamics of the state abruptly changes.
More formally, as one can see from \eqref{V_final},
the value function associated to Problem \ref{pb1}
will be a function of the total wealth of the agent.
In particular it will be a function of $g$ (see \eqref{DEF_g_h}).
Since, by Lemma \ref{LEMMA_EXISTENCE_g_h} the map $r \mapsto g'(r)$
has a discontinuity in $r=\tau_R$, it is clear that smoothness in time
will be lost by approaching directly Problem \ref{pb1} via a DP approach.
\end{rem}

\vspace{-0.3truecm}

\subsection{The evolution of admissible paths}
\label{SSE:EVOLUTION}

Fix the initial condition $(t,w,x) \in \mathcal{H}_+^{0, \infty}$ for system
\eqref{DYN_W_X_INFINITE_RETIREMENT_II} and $\pi \in \Pi(t,w,x)$.
Let $(W^{w,x}(\cdot;t,\pi), X^x(\cdot;t))$
be the corresponding solution.
In the following we will often use the shorthand notation
\vspace{-0.3truecm}
\begin{equation}
\label{Gamma_bar} \overline
\Gamma (s):=\Gamma(s,W^{w,x}(s;t,\pi), X^x(s;t)), \qquad s \ge t,
\vspace{-0.3truecm}
\end{equation}
and we will denote by
$\tau_t$ the first exit time of the process
$[t,+\infty)\times \Omega\to\R_+\times \mathcal{H}$,
$s \to (s,W^{w,x}(s;t,\pi), X^x(s;t))$ from
$\mathcal{H}_{++}^{0,\infty}$, that is
\vspace{-0.4truecm}
\begin{align}
\tau_t
 &:= \inf\left\{s\ge t:\; (s,W^{w,x}(s;t,\pi), X^x(s;t))
 \not\in \mathcal{H}_{++}^{0,\infty}\right\} \\
 &=\inf\{s\ge t :\;
(s,W^{w,x}(s;t,\pi), X^x(s;t)) \in \partial \mathcal{H}_{++}^{0, \infty}\}
\notag \\
&=\inf\{s\ge t:
\Gamma(s,W^{w,x}(s;t,\pi), X^x(s;t))=0\}.
\label{eq:deftaut}
\vspace{-0.4truecm}
\end{align}

\begin{lemma}\label{lm:Gammabar}
Fix the initial condition $(t,w,x) \in \mathcal{H}_+^{0, \infty}$
for system
\eqref{DYN_W_X_INFINITE_RETIREMENT_II} and $\pi \in \Pi_0$. Let
$(W^{w,x}(\cdot;t,\pi), X^x(\cdot;t))$ be the corresponding solution.
Then we have the following (where we write	$X_0(s)$ for $ X_0^x(s;t)$).
\begin{itemize}
\item[(i)]
The process $\overline\Gamma (s)$ in \eqref{Gamma_bar} satisfies,
for $s \ge t$, the SDE
\vspace{-0.3truecm}
\begin{align}
\label{eq:GammaProcessSDE} 	
%\begin{split} 	
&{\rm d}\overline\Gamma(s)=\left[(r+\delta)\overline\Gamma(s)-c(s)-\delta B(s)\right.
\notag \\
&
+ \left. \left(\theta^\top (s) \sigma +  g(s)  X_0(s)  \sigma_y^\top \right)
\kappa\right]{\rm d}s
\\ 	&
+ \left[\theta^\top(s) \sigma +  g(s)  X_0(s)  \sigma_y^\top
\right]\, {\rm d}Z(s).\notag
%\end{split} 	
\end{align}
\vspace{-0.3truecm}
\item[(ii)]
\vspace{-0.4truecm}
Assume that
$\Gamma(t,w,x)=0$, i.e. that $\overline \Gamma(t)=0$.
Then for every $s\ge t$ it must be
$\overline \Gamma(s)=0$, $\P-a.s.$,
and
\vspace{-0.3truecm}
\begin{equation}
\label{eq:zerostrategy}
 c(s,\omega)=0, \quad B(s,\omega)=0,
\quad \theta^\top(s) \sigma +  g(s)X_0(s)  \sigma_y^\top =0,
\vspace{-0.3truecm}
\end{equation}
must hold $ds\otimes \P-a.e. \; in \; [t,\infty) \times \Omega.$
Let now
$\overline\Gamma(t)>0$.
Then, for every $s\ge t$,
$\mathbb I_{\{\tau_t<s\}}\overline \Gamma(s)=0$, $\P-a.s.$,
and
\vspace{-0.3truecm}
\begin{align*}
&\mathbb I_{\{\tau_t<s\}}(\omega) c(s,\omega)=0, \quad \mathbb
I_{\{\tau_t<s\}}(\omega)B(s,\omega)=0, \\
&\mathbb
I_{\{\tau_t<s\}}\left[\theta^\top(s) \sigma +  g(s)   X_0(s)
\sigma_y^\top\right] =0,
\vspace{-0.3truecm}
\end{align*}
$ds\otimes \P-a.e. \; in \; [t,\infty) \times
\Omega$.
\end{itemize}
\end{lemma}

\vspace{-0.3truecm}

\begin{rem}
 We note  that the temporal dynamics of the total wealth process in \eqref{eq:GammaProcessSDE} and of admissible controls in \eqref{eq:zerostrategy} abruptly changes {at} $t= \tau_R$. In particular, for $t \ge\tau_R$, the time evolution of the total wealth process and of the admissible strategies  are as in the
classical Merton problem. For $t\in [0, \tau_R]$ we recover instead \cite[Lemma 4.9]{BGP}.
\end{rem}

\vspace{-0.4truecm}

\section{Solving the
infinite horizon Problem \ref{pb2}.}
\label{Vih_sec}

\vspace{-0.4truecm}

We start by Step (I) above dealing with the case
$t \in [\tau_R, \infty)$.
Note that, in this case, the dynamic of $W$ does not depend on $X$
and the functions $g$ and $h$ are null. Thus the constraint in
(\ref{CONSTRAINT_FIN_RET}) and consequently the set of admissible
controls do not depend on $X$: the value function depends just
on the state variable $W$. We then have to solve the following:
\begin{prob}
\label{pb2}
Let $t \in [\tau_R, +\infty)$ and $w \in \mathbb{R}_+$.
Consider the state equation
\vspace{-0.3truecm}
\begin{equation}
\label{DYNAMICS_WEALTH_INF_HOR}
\begin{cases}
{\rm d}W(s) =&  \left[ (r+\delta) W(s)+ \theta^\top(s) (\mu-r \mathbf 1)
  - c(s) - \delta B(s) \right]\, {\rm d}s +  \\
   & + \theta^\top (s)
   \sigma \, {\rm d}Z(s),\\
  W(t)\, =& w.
  \end{cases}
\vspace{-0.1truecm}
\end{equation}
Find the strategy $\overline \pi \in \Pi(t,w)$
that maximizes the objective functional
$J^{ih}(t,w;\pi):=$
\vspace{-0.3truecm}
\begin{eqnarray}\label{DEF_J_IH}
\mathbb E \left[
\int_{t}^{+\infty}  e^{-(\rho+ \delta) s }
 \left( \frac{(K c(s))^{1-\gamma}}{1-\gamma} +
  \delta \frac{\big(k B(s)\big)^{1-\gamma}}{1-\gamma}\right)
  \, {\rm d}s \right] \notag \\
\vspace{-0.6truecm}
\end{eqnarray}
recalling that, in this case, the constraint is simply
\vspace{-0.3truecm}
\begin{equation*}
W^{w}(s; t, \pi) \geq 0\quad \quad \mbox{for all}  \ s \geq t.
\vspace{-0.3truecm}
\end{equation*}
\end{prob}
The value function
$V^{ih}:[\tau_R, +\infty) \times {\mathbb{R}_+}
\rightarrow \overline{\mathbb{R}}$ is defined as
\begin{footnote}{Note that we allow at this point $V^{ih}$ to take
the values $\pm \infty$.}\end{footnote}
\vspace{-0.3truecm}
\begin{equation}
V^{ih}(t,w):= \sup_{\pi \in \Pi(t,w)}J^{ih}(t,w;\pi).
\vspace{-0.3truecm}
\end{equation}
Note that, by construction, it must be that
\vspace{-0.3truecm}
$$
V(t,w,x)=V^{ih}(t,w), \qquad \forall(t,w,x)\in
[\tau_R,+\infty)\times \mathcal{H}.
\vspace{-0.3truecm}
$$
This problem can be easily solved by appealing to
\cite[Theorem 5.1]{BGP}, as follows.

\begin{prop}
[Solution of Problem \ref{pb2}]
\label{PROPOSITION_MAIN_INF_HOR}
Consider the infinite-horizon optimization Problem \ref{pb2}
under Hypothesis \ref{hp:betarho}.
\vspace{-0.3truecm}
\begin{itemize}
\item The value function
$V^{ih}:[\tau_R,\infty) \times \mathbb{R}_+ \rightarrow
\overline{ \mathbb{R}}$ is given by
\vspace{-0.3truecm}
 \begin{equation}
 \label{Vfh}
 V^{ih}(t,w) =  e^{-(\rho + \delta )t} \hat \eta^{\gamma}
 \frac{  w^{1-\gamma} }
 {1-\gamma},
\vspace{-0.3truecm}
 \end{equation}
where
\vspace{-0.3truecm}
\begin{align}
\label{fbnu}
\hat \eta :=& \big( K^{-b}+\delta k^{-b}\big) \nu, \qquad
b= 1-\frac{1}{\gamma}, \\
\nu := &\frac{\gamma}{\rho + \delta -(1-\gamma)
 (r + \delta +\frac{\kappa^\top \kappa}{2\gamma })}. \notag
\vspace{-0.3truecm}
\end{align}
\end{itemize}
\begin{itemize}
\item The optimal strategies in feedback form are (here $s \ge t$)
\vspace{-0.3truecm}
\begin{align}\label{OPTIMAL_STRATEGIES_FIN_HOR}
%\begin{split}
& c^{*}(s):= K^{ -b }   \hat \eta^{-1}  W^*(s),
\quad
B^{*}(s):=k^{ -b } \hat \eta^{-1}  W^*(s), \notag
\\
& \theta^{*}(s):=   \frac{W^*(s) }{ \gamma}  (\sigma^\top)^{-1} \kappa  ,
%\end{split}
\vspace{-0.3truecm}
\end{align}
where the optimal wealth process $W^*$ has dynamics (here $s\ge t$)
\vspace{-0.3truecm}
\begin{align}\label{DYN_GAMMA*_PB1}
\begin{split}
& {\rm d} W^*(s) = \notag
\\
& = W^*(s) \Big(  r + \delta +\frac{1}{\gamma}
\kappa^\top \kappa - \nu^{-1}\Big)\, {\rm d}s +  \frac{W^*(s)}{\gamma } \kappa^\top \, {\rm d}Z(s),\notag \\ & \text{ with } W(t)=w.
\end{split}
\vspace{-0.3truecm}
\end{align}
%with $\big(c,B,\theta \big)$ as in (\ref{OPTIMAL_STRATEGIES_FIN_HOR}).
\end{itemize}
\end{prop}

\begin{rem}
\label{rm:newsec5}
Problem \ref{pb2} is a variant of the classical infinite horizon Merton problem taking into account the availability of life insurance. {Solution to Problem \ref{pb2} can also be found by  adapting classical arguments} (see e.g. \cite{Rogers}). Note, in particular, that one has to assume $\nu>0$, as required in Hypothesis \ref{hp:betarho}, in order for the Merton problem to have a solution.
If $\nu<0$, for example, one can prove that
an investor can achieve
infinite utility by delaying consumption when $\gamma \in(0,1)$
(see \cite[Section 1.6 and Proposition 1.3]{Rogers}
and \cite[Remark 2.5]{BGP}).
For the case $\gamma>1$, we refer to \cite{FreniGozziSalvadori06}, in which it is proved,
in a related deterministic case, that
$\nu<0$ implies $V^{ih}(t,w)= - \infty$.
\end{rem}

\vspace{-0.4truecm}

\section{Solving Problem \ref{pb3}: a finite horizon problem}
\label{Vfh_sec}

From the above Proposition \ref{PROPOSITION_MAIN_INF_HOR} we know that
$V(\tau_R,w,x)=V^{ih}(\tau_R,w)$ is given by \eqref{Vfh}.
Hence, we can now
maximize
the functional in \eqref{split} with $V(\tau_R,w,x)$ as terminal datum:
\begin{prob}
\label{pb3}
Let $t \in \left[0, \tau_R\right]$, and $(w,x)\in \mathcal{H}$
such that $\Gamma(t,w,x)\ge 0$.
Take the state equation
\vspace{-0.4truecm}
\begin{equation}
\label{pb3eq}
\begin{cases}
dW(s) =\\
 \left[W(s) r+ \theta^\top (s) (\mu-r\mathbf 1)
+ X_0(s) - c(s) + \delta \big(W(s)-B(s)\big) \right] {\rm d}s \\
+  \theta^\top(s) \sigma \, {\rm d}Z(s),
  \\
  dX(s) = A X(s) {\rm d}s + \big( C X(s)\big)_0^\top \, {\rm d}Z(s),
  \\
  W(t)= w,\quad X(t) = x.
 \end{cases}
 \vspace{-0.1truecm}
\end{equation}
Find the strategy $\overline \pi\in \Pi(t,w,x)$
that maximizes the objective functional
\vspace{-0.4truecm}
\begin{align*}
& J^{fh}(t,w,x;\pi)\!:= \\
& = \mathbb E \left[ \int_t^{\tau_R}
\!\! e^{-(\rho+ \delta) s }\! \left(\! \frac{( c(s))^{1-\gamma}}{1-\gamma}
 + \delta \frac{\big(k B(s)\big)^{1-\gamma}}{1-\gamma}\!\right)\! ds + \right.\\
+& \left. e^{-(\rho+\delta) \tau_R }
\frac{\hat \eta^{\gamma}
\big(W^{w,x}(\tau_R;t,\pi)\big)^{1-\gamma} }{1-\gamma } \right].
\end{align*}
Here, the constraint is:
\vspace{-0.3truecm}
\begin{equation*}
W^{w,x}\left(s;t, \overline \pi\right)+
g(s)X^x_0(s;t)+ \langle h(s), X^x_1(s;t) \rangle \geq 0
\,\quad \forall \ s \in [t, \tau_R].
\end{equation*}
\end{prob}
\begin{rem}
%\textcolor[rgb]{1.00,0.00,0.00}
{We observe that the above problem does not change
if we truncate the admissible trajectories at time $\tau_R$, i.e.,
if we take as set of admissible controls the following one,
which is the standard one for our finite horizon problem:}
\vspace{-0.3truecm}
\begin{align}
\label{DEF_PI_SECOND_DEFINITION_bis2}
& \Pi^{\tau_R}\left(t,w,x\right)
:= \notag\\
& \Big\{ \mathbb F\mbox{-predictable } \pi(\cdot)
\in L^1 (\Omega \times [t, \tau_R],\mathbb R^2_{+})
\times L^2(\Omega \times [t,\tau_R] , \mathbb R^n);
\notag
\\
& \mbox{such that: }
  W^{w,x}\left(s;t, \pi\right) +
g(s)X^x_0(s;t)+ \langle h(s), X^x_1(s;t) \rangle \geq 0\,
 \notag \\
& \forall s \in [t, \tau_R]\Big\}.
\vspace{-0.3truecm}
\end{align}
%\textcolor[rgb]{1.00,0.00,0.00}
{We will keep the less standard choice of the set of admissible control
because it allows us to simplify the notation.}
\end{rem}
We define the value function
$V^{fh}:{\mathcal{H}^{0, \tau_R}_{++}}\rightarrow {\mathbb{R}}$
as
\begin{eqnarray}
\label{VF}
 V^{fh}(t,w,x) :=\sup_{\pi \in   \Pi(t,w,x)}   J^{fh}(t,w,x;\pi).
\vspace{-0.4truecm}
\end{eqnarray}
To write the associated HJB equation we
rewrite system \eqref{pb3eq} in the unknowns $(W,X)$ as an equation in a single unknown $\mathcal{X}$.
We introduce two linear operators on $\mathcal{H} \times \mathbb{R}_+ \times \mathbb{R}_+ \times \mathbb{R}^n$: the unbounded operator $\mathcal{B}$ with values in $\mathcal{H}$ given by
\vspace{-0.4truecm}
\begin{equation*}
D(\mathcal{B})= \mathbb{R} \times D(A) \times \mathbb{R}_+ \times \mathbb{R}_+ \times \mathbb{R}^n,
\vspace{-0.4truecm}
\end{equation*}
\begin{equation*}
\mathcal{B}(w,x,c,B, \theta):=\left((r+\delta)w+\theta^\top(\mu-r\mathbf 1) +x_0-c-\delta B,Ax\right),
\end{equation*}
and the bounded operator $\mathcal{S}$ with values in
$L(\mathbb{R}^n;\mathcal{H})$ given by
\vspace{-0.3truecm}
\begin{equation*}
\mathcal{S}(w,x,c,B, \theta):= \left[z \rightarrow \left(\theta^\top \sigma z, (Cx)_0^\top z \right)\right].
\vspace{-0.3truecm}
\end{equation*}
It is not difficult to check that, for every fixed $(w,x,c,B, \theta)\in \mathcal{H} \times \mathbb{R}_+ \times \mathbb{R}_+ \times \mathbb{R}^n$, the adjoint of $\mathcal{S}(w,x,c,B, \theta)$ is the map $\mathcal{S}(w,x,c,B, \theta)^* \in L(\mathcal{H}; \mathbb{R}^n)$,
given by:
\vspace{-0.3truecm}
\begin{equation*} (u,p) \mapsto u\theta^\top\sigma + x_0p_0\sigma_y.
\vspace{-0.3truecm}
\end{equation*}
Let $\mathcal{I}$ denote the following space:
\vspace{-0.3truecm}
\begin{align*}
L^2(-d,0;\mathbb{R}) \times \left (L^2(-d,0;\mathbb{R}) \times L(L^2(-d,0;\mathbb{R}) ; L^2(-d,0;\mathbb{R}) )\right).
\vspace{-0.3truecm}
\end{align*}
Therefore $L(\mathcal{H},\mathcal{H}) \cong\mathcal{N}:=
\mathcal{H} \times \mathcal{H} \times \mathcal{I}$,
and given an element $P \in \mathcal{N}$,
we can index its entries as $P_{ij}$, where $i$ denotes the spaces
$\mathcal{H},\mathcal{H},\mathcal{I}$ (in this order) and $j$
the components in each space
(hence $(P_{11},P_{12},P_{13}) \in \mathcal{H}$).
Through this interpretation we can define the space of symmetric elements in $\mathcal{N}$ as
%\begin{equation*}
$\mathcal{N}_{sym}:=\{P\in \mathcal{N}:P_{ij}=P_{ji}, \ i,j=1,...3\}$. %\end{equation*}
By simple computations we then have, for any $P \in \mathcal{N}_{sym}$
 and $(w,x,c,B,\theta) \in \mathcal{H} \times \mathbb{R}_+ \times
 \mathbb{R}_+ \times \mathbb{R}^n$,
\vspace{-0.3truecm}
\begin{align}
\label{trace_term}
&\text{Tr}\left[P\mathcal{S}(w,x,c,B, \theta)\mathcal{S}(w,x,c,B,\theta)^* \right] =\\
& =  \theta^\top \sigma \sigma^\top\theta P_{11}+2\theta^\top\sigma \sigma_yx_0P_{12}+\sigma_y^\top\sigma_yx_0^2P_{22}. \notag
\vspace{-0.3truecm}
\end{align}
Hence, for any given function $f : [0, \tau_R] \times \mathcal{H} \rightarrow  \mathbb{R}$, its second Fr\'echet derivative, w.r.t. the variable in $\mathcal{H}$, at a given point $(t, w,x)$  is an element $\nabla^2f(t,w,x)\in \mathcal{N}_{sym}$. Formula \eqref{trace_term} then provides then the second order term that will appear in our HJB equation. Notice that, since the $L^2$-component of $\mathcal{S}$ is zero, such second order term turns out to be finite dimensional: it depends only on $P_{11}$, $P_{12}$ and $P_{22}$, that is only on the derivative w.r.t. the real components $(w, x_0)$. This is reasonable since the noise we consider affects explicitly only the dynamics of $W$ and $X_0$, but not that of $X_1$.
Let $t\in[0,\tau_R]$ and
$\pi(\cdot)=(c(\cdot), B(\cdot), \theta(\cdot))\in \Pi(t,w,x)$.
The state equation \eqref{pb3eq} is then rewritten as
\vspace{-0.3truecm}
\begin{equation}
\label{system_math}
\begin{cases}
{\rm d} \mathcal{X}(s)=\mathcal{B}(\mathcal{X}(s), \pi(s))\, {\rm d}s+ \mathcal{S}(\mathcal{X}(s),\pi(s))\, {\rm d}Z(s),\\
 \mathcal{X}(t)=(w,x).
\end{cases}
\end{equation}
%where solutions have to be intended in the mild sense.
By \cite[Chapter 5.6]{KARATZSAS_SHREVE_91} and Proposition
\ref{equiv_stoch} there is a unique mild solution to system
\eqref{system_math}. We will denote the solution at time
$t\in [0, \tau_R]$ by
$\mathcal{X}^{w,x}(s;t,\pi)=(W^{w,x}(s;t, \pi), X^x(s;t))$.
As usual, dependence on initial conditions will be at times subsumed in the following.

\subsection{The HJB equation and its explicit solution}

\vspace{-0.3truecm}

The Hamiltonian of our Problem \ref{pb3} is
\vspace{-0.2truecm}
\begin{equation*}
\tilde{\mathbb{H}}^{fh}: [0, \tau_R] \times (\mathbb{R} \times D(A)) \times \mathcal{H} \times \mathcal{N}_{sym} \rightarrow \mathbb{R} \cup \{\pm\infty\},
\vspace{-0.3truecm}
\end{equation*}
\vspace{-0.3truecm}
\begin{align*}
& \tilde{\mathbb{H}}^{fh}(t,(w,x),p,P)= \sup_{\pi \in \mathbb{R}^2_+
\times \mathbb{R}^n} \Big[\langle\mathcal{B}(w,x,\pi),p\rangle_{\mathcal{H}}\\
& + \frac 12 \text{Tr}\left[P\mathcal{S}(w,x,c,B, \theta)\mathcal{S}(w,x,c,B,\theta)^* \right] + l(t,\pi)\Big]
\vspace{-0.3truecm}
\end{align*}
where
\vspace{-0.3truecm}
\begin{equation*}
l(t,\pi):= e^{-(\rho+ \delta) t } \left( \frac{c^{1-\gamma}}{1-\gamma}
+ \delta \frac{k^{1-\gamma} B^{1-\gamma}}{1-\gamma}\right).
\end{equation*}
It is however convenient (see Remark \ref{rem:H_tilde} below) to rewrite the Hamiltonian in a more explicit way using the definitions of $\mathcal{B}$ and $\mathcal{S}$ and exploiting the definition of
$A^*$. Separating the part that depends on the controls from the rest, the Hamiltonian for Problem \ref{pb3} is the function:
\vspace{-0.3truecm}
\begin{equation*}
\mathbb{H}^{fh}: [0, \tau_R] \times \mathcal{H} \times (\mathbb{R} \times D(A^*))\times \mathcal{N}_{sym} \rightarrow \mathbb{R} \cup \{ \pm \infty\},
\vspace{-0.3truecm}
\end{equation*}
decomposed as:
\vspace{-0.3truecm}
\begin{align}\label{DEF_HAMILTONIAN_FIN_HOR}
&\mathbb H^{fh}(t,(w,x),p,P) :=
 \mathbb H^{fh}_0((w,x),p,P_{22}) +\notag \\
 &+\mathbb H^{fh}_{\max}(t,x_0,p_0,P_{11},P_{12})  ,
\end{align}
where:
\vspace{-0.3truecm}
\begin{align}\label{DEF_HAM_1 FIN_HOR}
& \mathbb H^{fh}_0((w,x),p,P_{22}):=
(r+\delta) w p_0 + x_0 p_0+  \langle x,A^*p_1  \rangle_{M_2}+ \notag \\
 & +\frac{1}{2}  \sigma_y^\top \sigma_y x_0^2 P_{22},
\vspace{-0.3truecm}
\end{align}
\vspace{-0.3truecm}
\begin{align}\label{DEF_HAM_max FIN_HOR}
\mathbb H^{fh}_{\max}(t,x_0,p_0,P_{11},P_{12})
:= \sup_{(c,B,\theta)\in \R^2_+\times  \mathbb R^n} \mathbb H^{fh}_{cv}
(t,x_0,p_0,P_{11},P_{12}; \pi),
\vspace{-0.3truecm}
\end{align}
\vspace{-0.3truecm}
\begin{align*}
&\mathbb H^{fh}_{cv}
(t,x_0,p_0,P_{11},P_{12};\pi)
:=  l(t,\pi)
 +[\theta^\top(\mu-r\mathbf 1) - c-\delta B]p_0
+ \\
&\frac{1}{2} \theta^\top  \sigma  \sigma^\top   \theta P_{11}
+  \theta^\top  \sigma \sigma_y x_0   P_{12}.
\vspace{-0.3truecm}
\end{align*}
Reordering the terms in the above definition of $\H^{fh}_{cv}$
we can write
\vspace{-0.3truecm}
\begin{align}
\label{DEF_H_CV_FIN_HOR}
& \H^{fh}_{cv}(t,x_0,p_0,P_{11},P_{12},\pi)=
 e^{-(\rho+ \delta) t }
\frac{c^{1-\gamma}}{1-\gamma}-cp_0  \notag \\
& +\delta\left[e^{-(\rho+ \delta) t }
 \frac{k^{1-\gamma} B^{1-\gamma}}{1-\gamma}-Bp_0\right]
 \notag\\
 &+
 \left[\frac12\left\vert\theta^\top\sigma\right\vert^2P_{11}
 +\theta^\top\sigma\sigma_y x_0P_{12}+\theta^\top(\mu-r\mathbf{1})p_0\right],
\vspace{-0.3truecm}
\end{align}
from which we easily see that for each $x_0\in\R$ and $P_{12}\in\R$
there are three possible cases:
\begin{enumerate}
\item\label{case1}
if $p_0>0$ and $P_{11}<0$ the sup in (\ref{DEF_HAM_max FIN_HOR})
is achieved at $\left(\theta^\ast,c^\ast,B^\ast\right)$, where
(here $b=1-\gamma^{-1}$),
\vspace{-0.3truecm}
\begin{align}
\label{MAXS_HAMILTONIAN_FIN_HOR}
& \theta^*=-\frac{(\sigma\sigma^\top)^{-1}}{P_{11}}
\left[(\mu-r\mathbf{1})p_0 +\sigma\sigma_yx_0P_{12}\right],
\; \\
& c^*=e^{-\frac{(\rho + \delta)t}{\gamma}}p_0^{-\frac{1}{\gamma}},
\; \notag
%\\
%&
\quad
B^*=k^{-b}e^{-\frac{(\rho +\delta)t}{\gamma}}p_0^{-\frac{1}{\gamma}}\notag
\vspace{-0.3truecm}
\end{align}
\item if $p_0<0$ or $P_{11}>0$ then the supremum in (\ref{DEF_HAM_max FIN_HOR}) is $+\infty$;
\item
if $p_0P_{11}=0$ the supremum in (\ref{DEF_HAM_max FIN_HOR}) can be finite or infinite depending on $\gamma$ and on the sign of the other terms involved.
\vspace{-0.3truecm}
\end{enumerate}
The HJB equation associated with Problem \ref{pb3} is the following PDE in the unknown $v: [0,\tau_R] \times \mathcal{H}\to\R$, with terminal condition
(here $\hat \eta$ as in \eqref{fbnu}):
\vspace{-0.3truecm}
\begin{eqnarray}
\label{HJB.}
\left\{
\begin{array}{l}
-  \partial_tv(t,w,x) = \mathbb H^{fh}\left(t,(w,x),\nabla v(t,w,x),\nabla^2v(t,w,x)\right)
\\
v(\tau_R,w,x) = e^{-(\rho+\delta)\tau_R}\hat \eta^{\gamma}\frac{ w^{1-\gamma}}{1-\gamma}.
\end{array}
\right.
\vspace{-0.3truecm}
\end{eqnarray}

\begin{defn}\label{DEF_SUPERSOLUTION_INF_RET}
A function $ v :  \mathcal{H}_{++}^{0,\tau_R} \longrightarrow \mathbb R $
is a \emph{classical solution}
of the HJB equation (\ref{HJB.}) if:
  \begin{enumerate}
\item
$v$ is continuously Fr\'{e}chet differentiable
and its second Fr\'echet derivatives with respect to the couple
$(w,x_0)$ exist and are continuous in $\H_{++}^{0,\tau_R}$.
\item \label{regularity}
$\partial_xv(t,w,x)\in D(A^*)$ for all
$(t,w,x)\in\mathcal{H}_{++}^{0, \tau_R}$
and $A^*\partial_xv(t,w,x)$ is continuous in $\mathcal{H}_{++}^{0, \tau_R}$.
\item
\vspace{-0.2truecm}
$v$ satisfies \eqref{HJB.} for every $(t,w,x)
\in  \mathcal{H}_{++}^{0, \tau_R}$.
\end{enumerate}
\end{defn}
\begin{rem}
\label{rem:H_tilde}
The difference between $\tilde{\mathbb{H}}^{fh}$ and $\mathbb{H}^{fh}$ lies in the term involving $A$, that appears as $\langle Ax,p_1\rangle$ in the former but as $\langle x,A^*p_1\rangle$ in the latter. This choice makes $\mathbb{H}^{fh}$ defined on the whole $\mathcal{H}_{++}^{0, \tau_R}$ instead than only on $\mathcal{H}_{++}^{0, \tau_R}\cap(\R\times\mathcal{D}(A))$, at the price of requiring further regularity of the solution, as specified in Definition \ref{DEF_SUPERSOLUTION_INF_RET}-~\ref{regularity}.
  In next Proposition \ref{PROPOSITION_GUESS_SOLVES_HJB_FIN_HOR}
  we provide an explicit solution that satisfies the required properties.
\end{rem}
If a solution $v$ to (\ref{HJB.}) satisfies $\partial_w v>0$ and $\partial^2_{ww}v<0$ uniformly in $(t,w,x)$, then we fall in case \ref{case1} above and, plugging $\theta^\ast,c^\ast,B^\ast$ in the definition of $\H^{fh}$, we find the PDE for $v$ to take the form
\vspace{-0.3truecm}
\begin{align}
\label{new_HJB}
-\partial_tv
= &
[(r+\delta) w + x_0]  \partial_w v
%\\
\notag
+   \langle x, A^*\partial_x v  \rangle_{M_2}
+ \\
& \frac{1}{2} \sigma_y^\top \sigma_y x_0^2 \partial^2_{x_0 x_0}v
 + \frac{\gamma}{1-\gamma} e^{-\frac{\rho+\delta}{\gamma}t}
  \big(1+\delta k^{-b}  \big)(\partial_w v)^{b}
\notag\\  \notag
&-\frac{1}{2} \frac{1}{\partial^2_{ww}v}
\left((\mu-r \mathbf 1)  \partial_wv
 +\sigma \sigma_y x_0
 \partial^2_{w x_0}v\right)^\top* \notag\\ &
*(\sigma\sigma^\top)^{-1}
\left((\mu-r \mathbf 1) \partial_wv+ \sigma \sigma_y
 \partial^2_{w x_0}v\right),
\vspace{-0.3truecm}
\end{align}
with the same terminal condition of (\ref{HJB.}).
We provide an explicit solution to \eqref{HJB.}.
\vspace{-0.3truecm}
\begin{prop}\label{PROPOSITION_GUESS_SOLVES_HJB_FIN_HOR}
Let Hypothesis \ref{hp:betarho} hold true. Let for $(t,w,x) \in \mathcal{H}_{++}^{0, \tau_R}$, \vspace{-0.4truecm}
\begin{equation} \label{EQ_GUESS_V}
\overline v(t,w,x):= F(t)^{\gamma} \frac{\Gamma^{1-\gamma}(t,w,x)}{1-\gamma}.
\vspace{-0.3truecm}
\end{equation}
The function $\overline v$ is a classical solution of the HJB equation \eqref{HJB.}, where
\vspace{-0.4truecm}
\begin{align}
\label{F.}
& F(t) := e^{-\frac{ (\rho + \delta)t}{\gamma}} f(t),\,
f(t):=  (\hat \eta - \eta)
\exp\left(-\frac{\tau_R-t}{\nu} \right) + \eta, \notag \\
&  \eta:=(1+ \delta k^{-b}) \nu,
\vspace{-0.3truecm}
\end{align}
$\hat \eta, \nu,b$ are given in \eqref{fbnu}, and
$\Gamma:=\Gamma(t,w,x)$ given in \eqref{Gamma}.
\end{prop}
\vspace{-0.3cm}

\begin{rem}\label{rem_boundary}
The function $\overline v$ can be defined also in ${\mathcal{H}}_{+}^{0, \tau_R}$
by setting,
on $\partial{\mathcal{H}}_{+}^{0, \tau_R}=\{\Gamma(t,w,x)=0\}$,
$\overline v(t,w,x)=0, \quad \hbox{when $\gamma\in (0,1)$},
\text{ and }
\overline v(t,w,x)=-\infty, \,\hbox{when $\gamma\in (1,+\infty)$}
$.
From now on we will consider $\overline v$ defined on
$\mathcal{H}_{+}^{0, \tau_R}$ with the above values at the boundary.
\end{rem}
\vspace{-0.3truecm}
Next, there is
the key step to get the optimal feedbacks.

\vspace{-0.3truecm}

\subsection{The fundamental identity}
%In this Section we use the function $\overline v$ to prove the
%fundamental identity (see \eqref{EXP_VER_INF_RET_III}
%in Proposition \ref{pr:fundid}).

We start with the following lemma.

\vspace{-0.3truecm}

\begin{lemma}
\label{LEMMA_LIMIT_AT_INFTY_GUESS_VALUE_FUNCTION_INFINITE_RETIREMENT}
Let $\overline v: \mathcal{H}_{+}^{0, \tau_R} \rightarrow
\overline{\mathbb{R}}$ be
the classical solution to the HJB \eqref{HJB.} given in Proposition
\ref{PROPOSITION_GUESS_SOLVES_HJB_FIN_HOR}.
Take any initial condition $(t,w,x)\in \mathcal{H}^{0, \tau_R}_{+}$
and take $\pi \in \Pi(t,w,x)$.
Then:
\vspace{-0.4truecm}
\begin{equation*}
\mathbb{E} \left[ \sup_{s \in [t,\tau_R]}\overline v
\left(s, (W^{w,x}(s;t, \pi), X^x(s;t))\right)\right] < \infty.
\vspace{-0.3truecm}
\end{equation*}
\end{lemma}
\begin{prop}
\label{pr:fundid}
Let $\overline v: \mathcal{H}_{+}^{0, \tau_R}
\rightarrow \overline{\mathbb{R}}$
be the classical solution to the HJB \eqref{HJB.} given in Proposition
\ref{PROPOSITION_GUESS_SOLVES_HJB_FIN_HOR}.
Let $(t,w,x) \in \mathcal{H}_{++}^{0, \tau_R}$ and take any $\pi \in \Pi(t,w,x)$ such that, when $\gamma>1$,
$J^{fh}(t,w,x;\pi)>-\infty$.
Let $\overline \tau_t:=\tau_t \wedge \tau_R$ (where
$\tau_t$ is defined in \eqref{eq:deftaut}).
Let  $\mathcal{X}(\cdot;t, \pi) =(W(\cdot;t,\pi), X(\cdot;t))$ be the trajectory corresponding to initial state $(t,w,x)$ and strategy $\pi$. The following identity holds:
\vspace{-0.3truecm}
\begin{align}\label{EXP_VER_INF_RET_III}
&\overline v\big(t,w,x)=J^{fh}(t,w,x;\pi) \\
& +\mathbb E  \int_t^{\overline \tau_t} \left[
\mathbb{H}^{fh}_{\max}\big(s,X_0(s),
\partial_w\overline v,
\partial_{ww} \overline v,
\partial_{wx_0}\overline v\big)\right. \notag
\\
&\left.
-\mathbb H^{fh}_{cv}\big(s,X_0(s),
\partial_w\overline v,
\partial_{ww} \overline v,
\partial_{wx_0}\overline v;
\pi(s) \big) \right]\,{\rm d}s. \notag
\vspace{-0.3truecm}
\end{align}
where we set
\vspace{-0.2truecm}
$$
\partial_w\overline v:=\partial_w\overline v(s,\mathcal{X}(s; t,\pi)),
\quad
\partial_{ww} \overline v:=\partial_{ww} \overline v(s,\mathcal{X}(s;t, \pi)),
\vspace{-0.2truecm}
$$
\vspace{-0.2truecm}
$$
\partial_{wx_0}\overline v:=\partial_{wx_0}\overline v(s,\mathcal{X}(s;t, \pi))
$$
\end{prop}

\vspace{-0.3truecm}

As a consequence of Proposition \ref{pr:fundid} we get the following.
\vspace{-0.3truecm}
\begin{cor}
\label{cr:FINITENESS_VALUE_FUNCTION}
The value function $V^{fh}$ given in \eqref{VF} is finite on ${\mathcal{H}}_{++}^{0, \tau_R}$ and $V^{fh}(t,w,x)\le \overline v(t,w,x)$ for every $(t,w,x)\in {\mathcal{H}}_{++}^{0, \tau_R}$.
\end{cor}

\vspace{-0.3truecm}

\subsection{Verification Theorem and optimal feedbacks}

Here we show that
$\overline v=V^{fh}$ in $\mathcal{H}_{+}^{0, \tau_R}$
also finding the optimal feedback strategies.
First, we provide the following definitions.

\vspace{-0.3truecm}

\begin{defn}\label{DEF_ADMISSIBLE_FEEDBACK STRATEGY_INF_RET}
Fix $(t,w,x) \in \mathcal{H}_{+}^{0, \tau_R}$.
A strategy $\overline \pi:=
\left(\overline c,\overline B, \overline \theta \right)$
is called an \textit{optimal strategy} if
$\left(\overline c,\overline B, \overline \theta \right)
\in \Pi\left(t,w,x\right)$,
and $V^{fh}(t,w,x)=J(t,w,x; \overline \pi)$.
\end{defn}

\vspace{-0.3truecm}

Note that if $(t,w,x)\in\partial \mathcal{H}_{+}^{0, \tau_R}$,
By Lemma \ref{lm:Gammabar} there is only one admissible strategy
which is also the unique optimal one.

\vspace{-0.3truecm}

\begin{defn}
We say that a function
$\left( \textbf{C}, \textbf{B}, \Theta\right):
\mathcal{H}_{+}^{0,\tau_R}\longrightarrow
\mathbb R_{+}\times \mathbb R_{+}\times \mathbb R^{n}$
is an \textit{optimal feedback map} if for any initial datum $(t,w,x)\in \mathcal{H}^{0, \tau_R}_{+}$, the closed loop equation
\vspace{-0.3truecm}
\begin{equation}
\label{eq:newCL}
\begin{cases}
{\rm d}W(s) =  \left[ (r+\delta) W(s)+\Theta^\top\left(s,W(s), X(s)\right) (\mu-r)  +  \right.
\\ \left.
\, + X_0(s) - \textbf{C}\left(s,W(s), X(s)\right) +
\left. - \delta\textbf{B}\left(s,W(s), X(s)\right)\right] \, {\rm d}s
  + \right.\\
  +  \Theta^\top \left(s,W(s), X(s)\right) \sigma \,{\rm d}Z(s),
  \\
  {\rm d}X(s)=AX(s)+(CX(s))^\top \, {\rm d}Z(s)
  \\
   (W(t),X(t))= (w,x).
 \end{cases}
\vspace{-0.1truecm}
 \end{equation}
has a unique solution $(W^*,X):= \mathcal{X}^*$
and the associated control strategy
$\left(\overline c, \overline B, \overline \theta \right)$
\vspace{-0.3truecm}
\begin{align}
\label{eq:new}
& \overline c(s):=  \textbf{C}\left(s,W^*(s),X(s)\right),
\quad
\overline B(s):=\textbf{B}\left(s, W^*(s),X(s)\right),
\\
& \overline \theta(s):=\Theta    \left(s, W^*(s),X(s)\right)
\vspace{-0.3truecm}
\end{align}
is an optimal strategy.
\end{defn}

\vspace{-0.3truecm}

As usual, the candidate optimal feedback map is given by the
maximum points of the Hamiltonian. In our case these are given by
\eqref{MAXS_HAMILTONIAN_FIN_HOR} so, putting there
$\partial_ w \overline v, \partial_{ww} \overline v,
\partial_{wx_0} \overline v$ in place of $p_0, P_{11}, P_{12}$
we get the map:
\vspace{-0.3truecm}
\begin{align}
\label{EQ_DEF_FEEDBACK_MAP}
\begin{split}
\left\{\begin{array}{l}
\textbf{C}_{f}(t,w,x):=   f(t)^{-1}  \Gamma(t,w,x)
\\
\textbf{B}_{f}(t,w,x):= k^{ -b } f(t)^{-1}  \Gamma(t,w,x)
\\
\Theta_f(t,w,x):= (\sigma\sigma^\top)^{-1}
\frac{(\mu-r\mathbf 1)}{\gamma}\Gamma(t,w,x)
 -  (\sigma^\top)^{-1}\sigma_y  g(t)x_0.
\end{array}\right.\end{split}
\vspace{-0.3truecm}
\end{align}
We prove that this is an optimal feedback map.
For $(t,w,x)\in {\mathcal{H}}_+^{0, \tau_R}$, denote with $W_f^*(s)$ the unique solution of the closed loop equation \eqref{eq:newCL}
with $\left(\textbf{C}_f, \textbf{B}_f, \Theta_f\right)$
in place of $\left(\textbf{C}, \textbf{B}, \Theta\right)$
%\begin{equation}
%\label{CLOSED_LOOP_EQUATION_W}
%\begin{cases}
%{\rm d}W(s) =  \left[ (r+\delta) W(s)+\Theta_f^\top\left(s,W(s), X(s)\right) (\mu-r)  +  X_0(s) - \textbf{C}_f\left(s,W(s), X(s)\right) \right.
%\\
%\left.\qquad \qquad  \quad- \delta\textbf{B}_f\left(s,W(s), X(s)\right)\right] \, {\rm d}s
%  +  \Theta_f^\top \left(s,W(s), X(s)\right) \sigma \,{\rm d}Z(s),
%  \\
%  {\rm d}X(s)=AX(s)+(CX(s))^\top \, {\rm d}Z(s)
%  \\
%    (W(t),X(t))= (w,x).
% \end{cases}
% \end{equation}
and set
\vspace{-0.2truecm}
\begin{equation}\label{DEF_GAMMA_INFTY_STAR}
\Gamma^*(s)= \Gamma\big(s,W_f^*(s), X(s)\big)  =W_f^*(s) + g(s)X_0(s)+ \langle h(s), X_1(s)\rangle.
\vspace{-0.2truecm}
\end{equation}
The control strategy associated with (\ref{EQ_DEF_FEEDBACK_MAP}) is then
\vspace{-0.2truecm}
\begin{align}\label{EQ_FEEDBACK_STRATEGIES_INF_RET}
&  \overline c_f(s):= \textbf{C}_f\left(s,W^*_f(s),X(s)\right) =f(s)^{-1}  \Gamma^*(s),  \\
&   \overline B_f(s):=\textbf{B}_f \left(s,W^*_f(s),X(s)\right) = k^{ -b } f(s)^{-1}  \Gamma^*(s),  \notag \\
&  \overline \theta_f(s):= \textbf{$\Theta$}_f
 \left(s,W^*_f(s),X(s)\right)=(\sigma\sigma^\top)^{-1}
 {(\mu-r\mathbf 1)}{\gamma^{-1}}\Gamma^*(s) \notag \\
 & \quad -(\sigma^\top)^{-1}\sigma_y  g(s)X_0(s). \notag
\end{align}
The next  Lemma ensures that this strategy is admissible.

\begin{lemma}
\label{36}
Let $(t,w,x) \in \mathcal{H}_+^{0, \tau_R}$. The process $\Gamma^*$ defined in \eqref{DEF_GAMMA_INFTY_STAR} satisfies the SDE
\vspace{-0.2truecm}
\begin{align}
\label{DYN_GAMMA*_PROPOSITION}
&{\rm d}  \Gamma^* (s) = \Gamma^* (s) \Big(  r + \delta +\frac{1}{\gamma}  |\kappa|^2 - f(s)^{-1} \big( 1+\delta k^{-b}\big) \Big)\, {\rm d}s \notag \\
& + \Gamma^* (s) \kappa^\top \sigma \, {\rm d}Z(s).
\vspace{-0.2truecm}
\end{align}
\end{lemma}

\begin{theorem}
\textbf{(Verification Theorem and Optimal feedback Map)}
\label{th:VERIFICATION_THEOREM_INF_RET}
The equality  $V^{fh}=\overline v$ in ${\mathcal{H}}_+^{0, \tau_R}$ holds true and  the function $\left(\textbf{C}_f, \textbf{B}_f, \Theta_f\right)$ defined in (\ref{EQ_DEF_FEEDBACK_MAP}) is an optimal feedback map. Finally, for every $(t,w,x)\in \mathcal{H}_{+}^{0, \tau_R}$ the strategy $\overline\pi_f:=(\overline c_f,\overline B_f,\overline\theta_f)$ is the unique optimal strategy.
\end{theorem}

\begin{cor}
\label{cor2}
For any $(t,w,x) \in \mathcal{H}_+^{0, \tau_R}$ the value function $V$ defined in \eqref{V_t} is equal to $V^{fh}$.
%has the explicit representation
%\begin{equation*}
%V(t,w,x)=F(t)^{\gamma} \frac{ \Gamma^{1-\gamma}(t,w,x)}{1-\gamma}
%,
%\end{equation*}
%with $F$ as in \eqref{F.}.
\end{cor}

%%%%%%%%%%%%%%%%%%%%%%%%%%%%%%%%%%%%%%%%%%%%%%%%%%%%%%%%%%%%%%%%%%%%%%%%%%%%%%%%%%%%%%%%%%%%%%%%%%%%%%%%%%%%

\vspace{-0.3truecm}

\section{Stating and interpreting of the main result}
\label{main_result_section}

In the present section we state and comment the main result
of the paper.

\begin{theorem}
\label{main_thm}
{Assume that Hypotheses \ref{hp:betarho} holds.}
\begin{itemize}
\item
The value function of Problem \ref{pb1}
(and hence of Problem \ref{p0} when $t=0$) is
\begin{equation}
\label{V_final}
V(t,w,x)= F(t)^{\gamma}\frac{\Gamma(t,w,x)^{1-\gamma}}{1-\gamma},
\, (t,w,x) \in \mathcal{H}_{+}^{0, \infty},
\vspace{-0.2truecm}
\end{equation}
where
\begin{equation*}
\vspace{-0.2truecm}
\Gamma(t,w,x):=w+g(t)x_0+ \langle h(t),x_1\rangle.
%\vspace{-0.2truecm}
\end{equation*}
with
\begin{equation*}
\begin{cases}
 g(t )=   \mathbb I_{\{t < \tau_R\}}\int_t^{\tau_R } e^{-\beta(\tau-t)}
 \big(h(\tau,0)+1 \big)\, {\rm d} \tau,
 \\
 h(t, s)=  \mathbb I_{\{t < \tau_R\}}\int^s_{(s+t-\tau_R)\vee (-d)}   e^{- (r+ \delta) (s-\tau)}* \\
 \quad\quad *g(t+s-\tau)\phi(\tau)\,{\rm d}\tau,
 \end{cases}
\vspace{-0.2truecm}
\end{equation*}
\begin{align*}
& f(t)=\left((\hat \eta - \eta)
\exp\left(  -\frac{(\tau_R-t)^+}{\nu} \right) + \eta \right),
\\
& F(t) := e^{-\frac{ (\rho + \delta)t}{\gamma}}f(t),
\end{align*}
and
 \begin{align*}
& \eta :=(1+ \delta k^{-b}) \nu
\qquad
\hat \eta:= (K^{-b}+ \delta k^{-b}) \nu, \\
& \nu  := \frac{\gamma}{\rho + \delta -(1-\gamma) (r + \delta +\frac{\kappa^\top \kappa}{2\gamma })},
\qquad b= 1-\frac{1}{\gamma}.
\end{align*}
\item The optimal strategies of Problem \ref{p0} (hence where $t=0$)
are, for all $s\ge  0$,
\vspace{-0.2truecm}
$$
\overline c_f(s):=
% \textbf{C}_f\left(t,W^*_f(t),X(t)\right)
K^{-bR(s)}	f(s)^{-1}  \Gamma^*(s),
\qquad
	\overline B_f(t):=
%\textbf{B}_f \left(t,W^*_f(t),X(t)\right) =
k^{ -b } f(s)^{-1}  \Gamma^*(s),
\vspace{-0.2truecm}
$$
\vspace{-0.2truecm}
	\begin{equation}\label{ASD3}
 \overline \theta_f(s):=
% \textbf{$\Theta$}_f\left(t,W^*_f(t),X(t)\right)=
%(\sigma\sigma^\top)^{-1} \frac{(\mu-r\mathbf 1)}{\gamma}\Gamma^*(s)
(
%\mcg{
\sigma\sigma^\top)^{-1} \frac{\kappa}{\gamma}\Gamma^*(s)
-
%\mcg{
g(s)y(s)(\sigma^\top)^{-1}\sigma_y  %g(s)y(s),
\vspace{-0.2truecm}
	\end{equation}
where
\vspace{-0.2truecm}
	\begin{equation}%\label{Gamma-infty2}
  \Gamma^*(s): =W^*(s) + g(s)y(s)+ \int_{-d}^0 h(s,\zeta) y(s+\zeta) \,  d\zeta,
  \, s \ge 0
%\vspace{-0.2truecm}
  	\end{equation}
denotes the optimal total wealth, with financial wealth $W^*(\cdot)$  given by the solution of equation (\ref{DYN_W_X_INFINITE_RETIREMENT_II}) with controls defined in (\ref{ASD3}) above, and with labor income $y(\cdot)$ given by the solution of the second equation in (\ref{DYNAMICS_WEALTH_LABOR_INCOME}).
\item The optimal total wealth process has dynamics	
\vspace{-0.2truecm}
	\begin{align}
	& {\rm d}  \Gamma^* (s) = \Gamma^* (s) \kappa^\top \sigma \,{\rm d}Z(s) + \\
&  \Gamma^* (s)
\Big(  r + \delta +\frac{1}{\gamma}  \kappa^\top \kappa
	- f(s)^{-1} \big(K^{-bR(s)}+\delta k^{-b}\big) \Big)\, {\rm d}s \notag	
\vspace{-0.2truecm}
	\end{align}
\end{itemize}
\end{theorem}
From the results above we see that the post-retirement problem preserves the original structure of Merton's solution.
Indeed, since $g,h$ are null after $\tau_R$,  the agent chooses fixed fractions of financial {wealth $W$} to consume, leave as bequest, and invest in the risky asset.
The pre-retirement solution departs from the Merton baseline in several ways. First, as in
\cite{bodie1992labor,CV},
 financial wealth is replaced by total wealth, as the agent capitalizes the value of future wages and treat it as a traded asset (human capital). Second, the optimal fractions of total wealth consumed and left as bequest are time varying, as they reflect the residual time to retirement (e.g., \cite{DYBVIG_LIU_JET_2010}). Third, the allocation to risky assets features a negative hedging demand arising
 from the exposure to market risk channelled by the labor income process (e.g., \cite{CV}).  In our model, however, this hedging demand presents novel features and gives rise to more articulate trade-offs. This is for two main reasons.  First, only the `future' component of human capital drives hedging demand, but not its `past' component  (see \cite{BGPZ}). Second, the annuity factor appearing in the hedging demand, $g=g_1+g_2$ (see \eqref{g1}, \eqref{g2}),  takes into account not only the market risk channeled by future wages ($g_1$), but also the compounding effect of their delayed contribution to labor income ($g_2$).

We can analyse these trade-offs more precisely by exploiting the closed form solutions obtained in Theorem \ref{main_thm}.
Let us first compare the optimal strategies above with the case in which labor income does not present any delay in its drift ($\phi=0$). As in \cite{BGP}, we observe
that the  dynamics of total wealth, $\Gamma^*$, is not influenced by the path dependent component of the model,
in the sense that,
{\it ceteris paribus}, changing $\phi$ (and hence $y$)
leaves the dynamics of  total wealth unchanged at each point in time.
Therefore, denoting by $\widetilde \Gamma$ the total wealth in case of $\phi=0$, we have that, for any initial point $(w, x)$ the following holds:
\vspace{-0.2truecm}
\begin{align}
& \Gamma^*(0)-\widetilde \Gamma(0)  =x_0 \left(g(0)-\frac{1-e^{-\beta \tau_R}}{\beta}\right) + \\
& \<h(0, \cdot),x_1(0)\> =x_0 g_2(0)+ \<h(0, \cdot),x_1(0)\>,
\vspace{-0.2truecm}
\end{align}
with $g_2$ defined in \eqref{g2}. %, where we note that under Hypothesis~\ref{HYP_BETA-BETA_INFTY} the quantity above is nonnegative.
For any $t\geq 0$, we can therefore write
\vspace{-0.2truecm}
\begin{align}
& \Gamma^*(t)-\widetilde \Gamma(t)= \left(\Gamma^*(0)-\widetilde \Gamma(0)\right)\times \notag \\
& *e^{\left(r +\delta +\frac{1}{\gamma}  \kappa^\top \kappa
	- f(t)^{-1} \left( K^{-bR(t)}+\delta k^{-b}\right)\right) t + \kappa^\top \sigma  Z(t)},
\vspace{-0.2truecm}
\end{align}
showing that any difference in total wealth, and hence optimal consumption level
$c^*$ and bequest target $B^*$, is shaped by the quantity $x_0 g_2(0)+ <h(0,\cdot),x_1(0)>$.
Let us now focus on the risky asset allocation,  and denote by
$\Theta_{f,\phi}$ the feedback map introduced in \eqref{EQ_DEF_FEEDBACK_MAP}. We can then write
\vspace{-0.3truecm}
\begin{align}
& \Theta_{f,\phi}(t,w,x) -\Theta_{f,0}(t,w,x)  =
(\sigma\sigma^\top)^{-1} \times \\
& \left[
\left(\frac{\kappa}{\gamma}
- \sigma_y \right)g_2(t)
x_0 + \frac{\kappa}{\gamma} <h(t, \cdot), x_1(t)>
\right].
\vspace{-0.3truecm}
\end{align}

This result offers a rich set of empirical predictions for stock market participation, as path dependency of labor income is shown to have
two complementary effects. On the one hand, it improves the predictability of labor income and increases the demand for risky assets via the `past' component of human capital,
$\<h(t), x_1(t)\>$. On the other hand, exposure to market risk is compounded by the delayed contribution of future wages, as quantified  by the annuity factor $g_2$. If the sensitivity of
labor income to market shocks is high enough
(i.e., $\sigma_y>\gamma^{-1}\kappa$), a
%\mcr{
negative
%reduced %negative
hedging demand arises,
but is counterbalanced by the positive contribution of the `past' component of human capital to the demand for risky assets. Moreover, as
the retirement date approaches, the annuity factor $g_2$ shrinks whereas the `past' component does not.
The model can therefore produce a rich variety of stock market participation
patterns.
To conclude, Figure~1 %\ref{fig:Hump}
offers some examples of risky asset allocations based on different parameter configurations consistent with the extant literature. We consider a single risky asset
and two types of labor income dynamics, one without delay (cases a1 and a2), and one with delay (cases b1 and b2). For the latter we assume $\phi=0.75\%$ and $d=5$.
We then consider two values for the labor income volatility parameter $\sigma_y$: a higher value of $10\%$ gives rise to a negative hedging demand (cases a1 and b1), whereas a lower value of $6\%$ is such that the condition $\sigma_y < \gamma^{-1}\kappa$ is satisfied.
The values of the market and preference parameters are based on  the contributions of  \cite{MP}, \cite{GM}, and \cite{BENZONI_ET_AL_2007}.
The examples depicted in Figure~1 shows that when a negative hedging demand is material (cases a1 and b1), the agent is initially short the risky asset and then gradually increases to a positive position as the bond-like nature of human capital becomes more pronounced with the approaching of the retirement date. A delay in labor income dynamics makes human capital more bond-like due to its predictable `past' component and raises the allocation to the risky asset. When the sensitivity of labor income to market shocks is small enough (cases a2 and b2), the allocation to the risky asset is positive and increasing throughout the agent's working life. Again, delayed labor income dynamics boost the allocation to the risky assset.

{\centering
\includegraphics[scale=0.9]{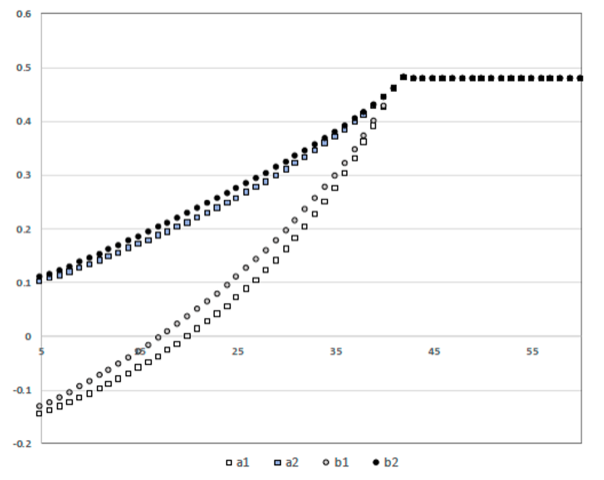}
}

\newpage
\appendix

\section{Proofs}

\begin{proof}[Proof of Proposition 3.1]
Existence and uniqueness of a continuous mild solution in $M_2$ of \eqref{INFINITE_DIMENSIONAL_STATE_EQUATION} and the equivalence between mild and weak solutions follow from \cite[Theorem 6.7]{DAPRATO_ZABCZYK_RED_BOOK} (see also \cite[Theorem 7.2]{DAPRATO_ZABCZYK_RED_BOOK}). In the same book, p. 160-161, the reader can find the precise definitions of mild and weak solution in this case. The equivalence property has been proven in Theorems 3.1-3.9 and Remark 3.7 in \cite{CHOJNOWSKA-MICHALIK_1978} (see also \cite{GOZZI_MARINELLI}, \cite{FabbriFederico14}).
\end{proof}

\begin{proof}[Proof of Proposition 3.2]
See \cite[Proposition 3.1]{BGP}.
\end{proof}

\begin{proof}of Lemma 3.3.
\textbf{Proof of (i).}
\\
By a change of variables let us rewrite (\ref{DEF_g_h}),
equivalently, as
\vspace{-0.2truecm}
\begin{equation}
\label{DEF_g_h_bis}
\begin{cases}
g(t )=   \mathbb I_{\{t < \tau_R\}}\int_t^{\tau_R } e^{-\beta(\tau-t)}
\big(h(\tau,0)+1 \big)\, {\rm d} \tau,   \  t \in[0, \infty),
\\
h(t,\zeta)=  \mathbb I_{\{t < \tau_R\}}\int_0^{(\tau_R-t)
\wedge(\zeta+d)}e^{- (r+ \delta) \tau}
g(t+\tau)\phi(\zeta-\tau)\,{\rm d}\tau,  \ \\
(t,\zeta)\in [0, \infty) \times [-d,0).
 \end{cases}
\vspace{-0.2truecm}
 \end{equation}
We apply the contraction principle depending on a parameter
(see \cite[Proposition C.11]{DapZab_ergodicity}).
Since $g$ and $h$ are identically zero for $t \ge \tau_R$
we work in the space
$\mathcal{M}:=C([0, \tau_R];\mathbb{R})
\times C([0, \tau_R]\times [-d,0];\mathbb{R})$, endowed with the norm
\begin{footnote}
{We endow $\mathcal{M}$ with this norm in order to make computations easier.}
\end{footnote}
\vspace{-0.5truecm}
\begin{equation*}
\|(\xi,\psi)\|_{\mathcal{M}}
= \left(\sup_{t \in [0, \tau_R]}|\xi(t)|^2
+\sup_{(t,\zeta) \in [0, \tau_R]\times [-d,0]}
|\psi(t,\zeta)|^2\right)^{\frac 12}.
\vspace{-0.3truecm}
\end{equation*}
We introduce the mapping
\vspace{-0.2truecm}
\begin{equation}\label{eq:Fdef}
 \mathcal{F}=(\mathcal{F}_1,\mathcal{F}_2):
 L^2(-d,0;\mathbb{R}) \times \mathcal{M} \rightarrow \mathcal{M}
\vspace{-0.2truecm}
 \end{equation}
\vspace{-0.2truecm}
 \begin{align*}
&(\phi, (g,h)) \mapsto \left(\int_t^{\tau_R } e^{-\beta(\tau-t)}
\big(h(\tau,0)+1 \big)\, {\rm d} \tau , \right.\\
&\left. \int_0^{(\tau_R-t)\wedge(\zeta+d)}e^{-(r+\delta) \tau}
g(t+\tau)\phi(\zeta-\tau)\,{\rm d}\tau \right). \notag
\vspace{-0.2truecm}
 \end{align*}
Notice that the term
$\mathcal{F}_2(\phi, (g,h))$
after a change of variables, is equal, on $[0,\tau_R]$,
to the r.h.s. of the second equation in \eqref{DEF_g_h}.
According to \cite[Proposition C.11]{DapZab_ergodicity}, if prove that
\vspace{-0.2truecm}
\begin{itemize}
\item [\textbf{(a)}] there exists $\alpha \in (0,1)$ such that  {for} $ (g,h), (\overline g, \overline h) \in \mathcal{M}$ the following inequality holds:
\vspace{-0.2truecm}
\begin{equation*}
\|\mathcal{F}(\phi, (g,h))-\mathcal{F}(\phi, (\overline g, \overline h))\|_{\mathcal{M}} \le \alpha \| (g,h)-(\overline g, \overline h)\|_{\mathcal{M}} \, ,
\vspace{-0.2truecm}
\end{equation*}
%\vspace{-0.2truecm}
\item [\textbf{(b)}]  and $\mathcal{F}$ is a continuous map,
\end{itemize}
\vspace{-0.2truecm}
then there exists a unique  map $G \in C(L^2(-d,0;\mathbb{R});\mathcal{M})$ such that
\vspace{-0.2truecm}
\begin{equation*}
\mathcal{F}(\phi, G(\phi))=G(\phi), \quad \phi \in L^2(-d,0;\mathbb{R})
\vspace{-0.2truecm}
\end{equation*}
and statement (i) immediately follows.

Let us prove (a). Take $(g,h)$ and $(\overline g, \overline h)$ in $\mathcal{M}$. We start with the following estimates: for $t \in [0, \tau_R]$, by means of H\"older's inequality, we infer
\vspace{-0.2truecm}
\begin{align*}
& \left | \mathcal{F}_1(\phi, (g,h))-\mathcal{F}_1(\phi, (\overline g,\overline h))\right|^2
= \\
&
\left|\int_t^{\tau_R }e^{-\beta(\tau-t)}
\big(h(\tau,0)-\overline h(\tau,0) \big)\, {\rm d} \tau \right|^2 \le
\\
& \int_t^{\tau_R}e^{2\beta (\tau-t)}\, {\rm d} \tau \int_t^{\tau_R}
|h(\tau,0)-\overline h(\tau,0)|^2\, {\rm d}\tau \le\\
& C  \sup_{t \in [0, \tau_R]}
\sup_{\zeta \in [-d,0]}|h(t,\zeta)-\overline h(t,\zeta)|^2,
\vspace{-0.2truecm}
 \end{align*}
where $C$ is a positive constant depending on $\beta$ and $\tau_R$.
Let $(t,\zeta)\in[0, \tau_R] \times [-d,0]$.
Thanks to H\"older's inequality we get
\vspace{-0.3truecm}
 \begin{align*}
&| \mathcal{F}_2(\phi, (g,h))-\mathcal{F}_2(\phi,
(\overline g,\overline h))|^2
=\\
& \left| \int_0^{(\tau_R-t)\wedge(\zeta+d)}e^{-(r+ \delta) \tau}
(g(t+\tau)-\overline g(t+\tau))\phi(\zeta-\tau)\,{\rm d}\tau\right|^2\le
\\
& \sup_{\tau \in [0,(\tau_R-t)\wedge(\zeta+d)]}
e^{2(r+ \delta) \tau}
 \int_0^{(\tau_R-t)\wedge(\zeta+d)}
|\phi(\zeta-\tau)|^2\, {\rm d}\tau \\
&\times
\int_0^{(\tau_R-t)\wedge(\zeta+d)}|g(t+\tau)-\overline g(t+\tau)|^2
\, {\rm d}\tau \leq
\\
&  \tilde C \sup_{t \in [0, \tau_R]}|g(t)-\overline g(t)|^2,
\\[-0.8truecm]
 \end{align*}
where $\tilde C$ is a positive constant depending on $\tau_R$, $d$, $r, \delta$ and $\|\phi\|_{L^2}$.
If we now denote by $\overline C=C\vee\tilde C$, putting together the last two estimates we obtain, for every
$(t,\zeta)\in[0, \tau_R] \times [-d,0]$,
\vspace{-0.2truecm}
\begin{align*}
&\left| \mathcal{F}_1(\phi, (g,h))-\mathcal{F}_1(\phi, (\overline g,\overline h))\right|^2 + | \mathcal{F}_2(\phi, (g,h))-\mathcal{F}_2(\phi, (\overline g,\overline h))|^2
\\
&\le \overline C \left(\sup_{t \in [0, \tau_R]}|g(t)-\overline g(t)|^2 + \sup_{t \in [0, \tau_R]}\sup_{\zeta \in [-d,0]}
|h(t,\zeta)-\overline h(t,\zeta)|^2\right).
\vspace{-0.2truecm}
\end{align*}
Taking the supremum on $[0,\tau_R] \times [-d,0]$ in the l.h.s. we get
\vspace{-0.2truecm}
\begin{equation*}
\|\mathcal{F}(\phi, (h,g))-\mathcal{F}(\phi, (\overline g, \overline h))\|^2_{\mathcal{M}} \le \overline C \| (g,h)-(\overline g, \overline h)\|^2_{\mathcal{M}}.
\vspace{-0.2truecm}
\end{equation*}
One can easily verify that, if $\tau_R$ is small enough, then $\overline C <1$ and then (a) is verified. On the other hand, if $\tau_R$ is such that $\overline C\ge 1$, one can iterate the argument in the intervals $[0,T]$, $[T, 2T]$, etc, with $0<T<\tau_R$  such that $\overline C(T)<1$. Thus statement (a) is proved.

\smallskip

Let us now prove statement (b). Let $(\phi, (g,h))$ and $(\overline \phi,(\overline g, \overline h))$ belong to $L^2(-d,0;\mathbb{R}) \times \mathcal{M}$.  With the same computations as before we write
\vspace{-0.2truecm}
\begin{align*}
&\|\mathcal{F}(\phi, (h,g))-\mathcal{F}(\overline \phi, (\overline g, \overline h))\|^2_{\mathcal{M}} \le \\
& c_1  \sup_{t \in [0, \tau_R]}\sup_{\zeta \in [-d,0]}|h(t,\zeta)-\overline h(t,\zeta)|^2 +
\\
& \sup_{t \in [0, \tau_R]}\sup_{\zeta \in [-d,0]} \left| \int_0^{(\tau_R-t)\wedge(\zeta+d)}e^{-(r+\delta) \tau}  \times \right.\\
&\left.\left(g(t+\tau)\phi(\zeta-\tau)-\overline g(t+\tau))\overline\phi(\zeta-\tau)\right)\,{\rm d}\tau\right|^2.
\vspace{-0.2truecm}
\end{align*}
Adding and subtracting the term $\int_0^{(\tau_R-t)\wedge(s+d)}   e^{- (r+ \delta) \tau} g(t+\tau)\overline \phi(s-\tau)\,{\rm d}\tau$, by means again of H\"older's inequality, we can estimate
(here and below $\|g\|_\infty:=\sup_{t \in [0, \tau_R]}|g(t)|$)
\vspace{-0.2truecm}
\begin{align*}
&\left| \int_0^{(\tau_R-t)\wedge(\zeta+d)}   e^{- (r+ \delta) \tau} \left(g(t+\tau)\phi(\zeta-\tau)-\overline g(t+\tau))\overline\phi(\zeta-\tau)\right)\,{\rm d}\tau\right|^2
\\
&
\le 2 \left( \left| \int_0^{(\tau_R-t)\wedge(\zeta+d)}
e^{- (r+ \delta) \tau}g(t+\tau)(\phi-\overline \phi)(\zeta-\tau)
\,{\rm d}\tau\right|^2+ \right. \\
& \left. \left| \int_0^{(\tau_R-t)\wedge(\zeta+d)}
e^{- (r+ \delta) \tau}(g-\overline g)(t+\tau)
\overline \phi(\zeta-\tau)\,{\rm d}\tau\right|^2\right)
\\
& \le c_2 \|g\|_\infty
\|\phi-\overline \phi\|^2_{L^2([-d,0];\mathbb{R})}
+c_3 \|\overline \phi\|^2_{L^2([-d,0];\mathbb{R})}
\sup_{t \in [0 \tau_R]}|g(t)-\overline g(t)|^2.
\vspace{-0.2truecm}
\end{align*}
If we now denote by $\tilde C_1$ the maximum between $c_1$, $c_2\|g\|_\infty$ and $c_3 \|\overline \phi\|^2_{L^2([-d,0];\mathbb{R})}$, collecting the above estimates, we obtain
\vspace{-0.2truecm}
\begin{align*}
&\|\mathcal{F}(\phi, (h,g))-\mathcal{F}(\overline \phi, (\overline g, \overline h))\|^2_{\mathcal{M}}
\le \\
& \tilde C_1 \|(\phi, (g,h))-(\overline \phi, (\overline g, \overline h))\|^2_{L^2([-d,0];\mathbb{R})\times \mathcal{M}},
\vspace{-0.2truecm}
\end{align*}
and statement (b) then follows from straightforward considerations.

Assume now that $\phi$ is a.e. positive. Then, calling
\vspace{-0.2truecm}
$$
\mathcal{M}^+:=C([0, \tau_R];\mathbb{R}_+) \times C([0, \tau_R]\times [-d,0];\mathbb{R}_+)\subseteq \mathcal{M}
\vspace{-0.2truecm}
$$
it is immediate to see, from its definition, that the map $\mathcal{F}(\phi,\cdot)$ brings $\mathcal{M}^+$ into itself. Hence, from the proof above it must have a unique fixed point in $\mathcal{M}^+$, which, by uniqueness, must coincide with the fixed point found above in $\mathcal{M}$.

\smallskip

\textbf{Proof of (ii).}
The fact that $g\in C^1([0\tau_R];\R)$ immediately follows from the fact that $h$ is continuous, the first of (\ref{DEF_g_h}), and the Torricelli Theorem. It is clear, by differentiating, that the function $g$ solves the first equation in (\ref{CAUCHY_PROBLEM_g_h}). The last three equation of (\ref{CAUCHY_PROBLEM_g_h}) are obviously true by \eqref{DEF_g_h}. Let us then verify that $h$ solves the second equation in (\ref{CAUCHY_PROBLEM_g_h}). Let $\tau_R-t<\zeta+d$. Differentiating the second of \eqref{DEF_g_h_bis} and using that $g$ is continuous and $g(\tau_R)=0$ we have
\vspace{-0.2truecm}
\begin{align}
%\label{PART_DIFF_h_t}
\frac{ \partial h(t,\zeta)}{\partial t }
&=
%-e^{- (r+ \delta)(\tau_R-t)}g(\tau_R)\phi(\zeta-\tau_R+t)
\int_0^{\tau_R-t}   e^{- (r+ \delta) \tau}
  g^{\prime}(t+\tau)\phi(\zeta-\tau)\,{\rm d}\tau
%\nonumber
%\\
%&=\int_0^{\tau_R-t}   e^{-(r+ \delta) \tau}
%  g^{\prime}(t+\tau)\phi(\zeta-\tau)\,{\rm d}\tau
\label{PART_DIFF_h_t}
\vspace{-0.2truecm}
\end{align}
and
\vspace{-0.2truecm}
\begin{equation}\label{PART_DIFF_h_s}
\frac{ \partial h(t,\zeta)}{\partial \zeta}
= \int_0^{\tau_R-t}   e^{- (r+ \delta) \tau}
  g(t+\tau)\phi^{\prime}(\zeta-\tau)\,{\rm  d}\tau.
\vspace{-0.2truecm}
\end{equation}
Integration by parts formula yields
\vspace{-0.2truecm}
\begin{align*}
& - (r+ \delta)\int_0^{\tau_R-t}   e^{- (r+ \delta) \tau}
  g(t+\tau)\phi(\zeta-\tau) \,{\rm d}\tau
  =\\
& e^{- (r+ \delta)( \tau_R-t)}
  g(\tau_R)\phi(\zeta-\tau_R+t)
   -   g(t)\phi(\zeta)-
   \\
 &  \int_0^{\tau_R-t}   e^{- (r+ \delta) \tau}
  g^{\prime}(t+\tau)\phi(\zeta-\tau)\,{\rm d}\tau
  + \\
 & \int_0^{\tau_R-t}   e^{- (r+ \delta) \tau}
  g(t+\tau)\phi^{\prime}(\zeta-\tau) \,{\rm d}\tau.
\vspace{-0.2truecm}
\end{align*}
Recalling the definition of $h$ in \eqref{DEF_g_h_bis} and substituting (\ref{PART_DIFF_h_t}) and (\ref{PART_DIFF_h_s}) in the above expression, we get
\vspace{-0.2truecm}
\begin{equation*}
 -(r+\delta)h(t,\zeta)
=- \frac{ \partial h(t,\zeta)}{\partial t}
+\frac{ \partial h(t,\zeta)}{\partial \zeta} -g(t) \phi(\zeta).
\vspace{-0.2truecm}
\end{equation*}
Thus we immediately see that for $\tau_R-t<\zeta+d$ the function $h$ solves the second equation of (\ref{CAUCHY_PROBLEM_g_h}). Let $\tau_R-t>\zeta+d$. We have
\vspace{-0.2truecm}
\begin{equation}\label{PART_DIFF_h_tbis}
\frac{ \partial h(t,\zeta)}{\partial t }=
\int_0^{\zeta+d}   e^{- (r+ \delta) \tau}
  g^{\prime}(t+\tau)\phi(\zeta-\tau)\,{\rm d}\tau,
\vspace{-0.2truecm}
\end{equation}
and
\vspace{-0.2truecm}
\begin{align}\label{PART_DIFF_h_sbis}
& \frac{ \partial h(t,\zeta)}{\partial \zeta}
=
-e^{- (r+ \delta)(\zeta+d)}  g(t+\zeta+d)\phi(-d)
+
\\
& \int_0^{\zeta+d}   e^{- (r+ \delta) \tau}
  g(t+\tau)\phi^{\prime}(\zeta-\tau)\,{\rm  d}\tau. \notag
\vspace{-0.2truecm}
\end{align}
Integrating by parts as in the previous case shows that, also
for $\tau_R-t>\zeta+d$ the function
$h$ solves the second equation of (\ref{CAUCHY_PROBLEM_g_h}).
Moreover, since $\phi\in C^1([-d,0];\R)$ with $\phi (-d)=0$
we immediately see that, also when $\tau_R-t=\zeta+d$ the function $h$ is $C^1$ and satisfies the second equation of (\ref{CAUCHY_PROBLEM_g_h}).

\smallskip

\textbf{Proof of (iii).}
Let now $\phi \in L^2(-d,0;\R)$. The function $g$ belongs to $C^1([0\tau_R];\R)$ by the same argument of point $(ii)$.
First we recall that, by point (i) $h$ is continuous, hence bounded
in $[0,\tau_R]\times[-d,0]$, so, in particular,
we have $h(t,\cdot) \in L^2 ([-d,0], \mathbb R)$ for all $t\in[0,\tau_R]$.
%Indeed, for some $C>0$ (possibly changing from line to line) we have
%\begin{align*}
%\int_{-d}^0 h^2(t,\zeta) \,{\rm d}\zeta =&
%\int_{-d}^0 \Big( \int_0^{(\tau_R-t)\wedge(\zeta+d)}e^{-(r+ \delta) \tau}
%  g(t+\tau)\phi(\zeta-\tau)\,{\rm d}\tau\Big)^2 \,{\rm d}\zeta \\
%%
%\le &  C\int_{-d}^0   \int_0^{(\tau_R-t)\wedge(s+d)}  e^{- 2(r+ \delta) \tau}
%  g^2(t+\tau)\phi^2(s-\tau)\, {\rm d}\tau \, {\rm d}s
%\\
%%
%\le & C g^2_{max} \int_{-d}^0   \int_0^{(\tau_R-t)\wedge(s+d)}
%\phi^2(s-\tau)\, {\rm d}\tau \, {\rm d}s
%\le C d \ g^2_{max} \|\phi  \|_{L^2(-d,0;\mathbb{R})}^2 <+\infty.
%\end{align*}
We observe that the derivative $\frac{ \partial h(t,s)}{\partial t}$, defined above in the proof of (ii), makes still sense and defines a function
\vspace{-0.2truecm}
\begin{align*}
&[0,\tau_R] \to L^2(-d,0;\R), \\
&  t \mapsto \overline h(t,\cdot):=\int_0^{(\tau_R-t)\wedge (\cdot+d)}   e^{- (r+ \delta) \tau} g^{\prime}(t+\tau)\phi(\cdot-\tau)\,{\rm d}\tau.
\vspace{-0.2truecm}
\end{align*}
Using the fact that translations are continuous in the $L^2$ norm
we get that such function is continuous.
It remains to prove that it is the derivative of the map
$[0,\tau_R] \to L^2(-d,0;\R)$, $t \mapsto h(t,\cdot)$,
i.e. that the following limit holds in the $L^2$ norm
\vspace{-0.2truecm}
$$
\lim_{\alpha\to 0}
\frac{h(t+\alpha,\cdot)- h(t,\cdot)}{\alpha}
= \frac{ \partial h(t,\cdot)}{\partial t}.
\vspace{-0.2truecm}
$$
Let $\zeta\in[-d,0]$, $t\in(0,\tau_R)$ and $\alpha>0$. Since $\alpha$ must tend to $0$ we can suppose, without loss of generality that $t+\alpha\in(0,\tau_R)$. We have to study two cases.

First we suppose that $(\tau_R-t)<(\zeta+d)$ and $(\tau_R-(t+\alpha))<(\zeta+d)$. Then
\vspace{-0.2truecm}
\begin{align}
\label{eq:alfa1}
&\frac{h(t+\alpha,\zeta)-h(t,\zeta)}{\alpha}=
\\
\notag
&=\frac{1}{\alpha}\left[\int_0^{\tau_R-(t+\alpha)}e^{-(r+\delta) \tau}
g(t+\alpha+\tau)\phi(\zeta-\tau)\,{\rm d}\tau \right.
-\notag\\
& \left.\int_0^{\tau_R-t}   e^{- (r+ \delta) \tau}
g(t+\tau)\phi(\zeta-\tau)\,{\rm d}\tau\right]=
\\
\notag
&= \frac{1}{\alpha}\int_{\tau_R-(t+\alpha)}^{\tau_R-t}
e^{-(r+ \delta) \tau}
g(t+\alpha+\tau)\phi(\zeta-\tau)\,{\rm d}\tau
+ \notag\\
& \int_0^{\tau_R-(t+\alpha)}   e^{- (r+ \delta) \tau}
\frac{g(t+\alpha+\tau)-g(t+\tau)}{\alpha}\phi(\zeta-\tau)\,{\rm d}\tau.
\vspace{-0.2truecm}
\end{align}
Recalling that $g$ is continuous and $g(t)=0$ for $t\ge\tau_R=0$,
the first integral above is $0$.

Moreover, since $g$ is continuous differentiable, we get
\vspace{-0.2truecm}
\begin{align*}
& \lim_{\alpha\to0}\int_0^{\tau_R-(t+\alpha)}   e^{- (r+ \delta) \tau}
\frac{g(t+\alpha+\tau)-g(t+\tau)}{\alpha}\phi(\zeta-\tau)
\,{\rm d}\tau= \\
& \int_0^{\tau_R-t}   e^{- (r+ \delta) \tau}
g'(t+\tau)\phi(\zeta-\tau)\,{\rm d}\tau.
\vspace{-0.2truecm}
\end{align*}
Now we study the case when $(\tau_R-t)>(\zeta+d)$
 and $(\tau_R-(t+\alpha))>(\zeta+d)$. We have
\vspace{-0.2truecm}
\begin{align}
\label{eq:alfa2}
&\frac{h(t+\alpha,\zeta)-h(t,\zeta)}{\alpha}
=
\\
\notag
&=\frac{1}{\alpha}\left[\int_0^{\zeta+d}   e^{- (r+ \delta) \tau} g(t+\alpha+\tau)\phi(\zeta-\tau)\,{\rm d}\tau - \right. \notag\\
& \left. \int_0^{\zeta+d}   e^{- (r+ \delta) \tau}
g(t+\tau)\phi(\zeta-\tau)\,{\rm d}\tau\right]=
\\
\notag
&=\int_0^{\zeta+d}   e^{- (r+ \delta) \tau}
\frac{g(t+\alpha+\tau)-g(t+\tau)}{\alpha}\phi(\zeta-\tau)\,{\rm d}\tau.
\vspace{-0.2truecm}
\end{align}
Recalling that $g$ is continuous differentiable we get
\vspace{-0.2truecm}
\begin{align*}
& \lim_{\alpha\to0}\int_0^{\zeta+d}   e^{- (r+ \delta) \tau}
\frac{g(t+\alpha+\tau)-g(t+\tau)}{\alpha}\phi(\zeta-\tau)
\,{\rm d}\tau= \\
& \int_0^{\zeta+d}   e^{- (r+ \delta) \tau}
g'(t+\tau)\phi(\zeta-\tau)\,{\rm d}\tau.
\vspace{-0.2truecm}
\end{align*}
By combining the two cases above we obtain:
\vspace{-0.2truecm}
\begin{align}
\label{lim_rapp_incre_h}
& \lim_{\alpha\to0} \frac{h(t+\alpha,\zeta)-h(t,\zeta)}{\alpha}= \\
&\int_0^{(\tau_R-t)\wedge(\zeta+d)}   e^{- (r+ \delta) \tau} g'(t+\tau)\phi(\zeta-\tau)\,{\rm d}\tau,\ \text{for a.e. }\zeta\in[-d,0]. \notag
\vspace{-0.2truecm}
\end{align}
The claim follows if we prove that the limit above holds in
$L^2([-d,0])$. To do this we apply dominated convergence theorem.
Indeed, calling $\|g'\|_\infty:=\sup_{t\in [0,\tau_R]}|g'|$, we get,
from \eqref{eq:alfa1}-\eqref{eq:alfa2},
\vspace{-0.2truecm}
\begin{align*}
&\frac{h(t+\alpha,\zeta)-h(t,\zeta)}{\alpha}
\le \\
& \int_0^{(\tau_R-t)\wedge(\zeta+d)}\!\!\!\!\!  e^{-(r+ \delta) \tau}
\left|\frac{g(t+\alpha+\tau)-g(t+\tau)}{\alpha}\right|
|\phi(\zeta-\tau)|\, {\rm d}\tau \\
& \le \|g'\|_\infty \int_0^{\zeta+d} |\phi(\zeta-\tau)|\, {\rm d}\tau .
\vspace{-0.2truecm}
\end{align*}
Since $\phi\in L^2(-d,0)$ the map
$\zeta\mapsto \int_0^{\zeta+d} |\phi(\zeta-\tau)|\, {\rm d}\tau$
belongs to $L^2(-d,0;\R)$ and the claim follows.
%Therefore by applying dominated convergence theorem we get
%\[s\mapsto \int_0^{(\tau_R-t)\wedge(s+d)}   e^{- (r+ \delta) \tau}
%g'(t+\tau)\phi(s-\tau)\,{\rm d}\tau\in L^2(-d,0;\R)\]
%and
%\begin{align*}
%L^2-\lim_{\alpha\to0}\frac{h(t+\alpha,\cdot)- h(t,\cdot)}{\alpha}=&L^2-\lim_{\alpha\to0}\int_0^{(\tau_R-t)\wedge(\cdot+d)}   e^{- (r+ \delta) \tau}
%\frac{g(t+\alpha+\tau)-g(t+\tau)}{\alpha}\phi(s-\tau)\,{\rm d}\tau\\
%=&\int_0^{(\tau_R-t)\wedge(\cdot+d)}   e^{- (r+ \delta) \tau}
%g'(t+\tau)\phi(s-\tau)\,{\rm d}\tau=\frac{\partial h}{\partial t}(t,\cdot).
%\end{align*}

\smallskip
\textbf{Proof of (iv).}
%\black{CHECK CAREFULLY THIS PROOF}
%\begin{enumerate}[label=\roman{*}., ref=(\roman{*})]
%Fix $t \ge 0$. Let us check $\big(g(t),h(t, \cdot)\big) \in \mathcal D(A^*)$.
%Indeed, for some $C>0$ (possibly changing from line to line) we have
% %\begin{footnote}{If functions $a,b \ge 0$ satisfy the inequality $a \le Cb$, with a constant $C>0$,we write $a \lesssim b$.}\end{footnote}
%\begin{align*}
%\int_{-d}^0 h^2(t,s) \,{\rm d}s =&
%\int_{-d}^0 \Big( \int_0^{(\tau_R-t)\wedge(s+d)}   e^{- (r+ \delta) \tau}
%  g(t+\tau)\phi(s-\tau)\,{\rm d}\tau\Big)^2 \,{\rm d}s \\
%%
%\le &  C\int_{-d}^0   \int_0^{(\tau_R-t)\wedge(s+d)}  e^{- 2(r+ \delta) \tau}
%  g^2(t+\tau)\phi^2(s-\tau)\, {\rm d}\tau \, {\rm d}s
%\\
%%
%\le & C g^2_{max} \int_{-d}^0   \int_0^{(\tau_R-t)\wedge(s+d)}
%\phi^2(s-\tau)\, {\rm d}\tau \, {\rm d}s
%\le C d \ g^2_{max} \|\phi  \|_{L^2(-d,0;\mathbb{R})}^2 <+\infty.
%\end{align*}
%Observe now that, by point (iii) above, $h(t,\cdot)$
%is differentiable, thus it is in
%belongs to $W^{1,2}\big([-d,0]; \mathbb R \big)$,
%
%Let $0 \leq t \leq \tau_R$,
%and denote with $g_{max}$ the maximum of $g$ on $[0,\tau_R]$,
First, by the second of \eqref{DEF_g_h} we have $h(t,-d)=0$ for all $t\ge 0$. Thus, to prove that, given any $t\ge 0$,
$\big(g(t),h(t, \cdot)\big)$ is in $\mathcal D(A^*)$.  it is enough to show that $h(t,\cdot) \in W^{1,2} ([-d,0], \mathbb R)$.
For $t\ge \tau_R$ this is clear since in this case $g=h=0$.
Let then $t\in[0,\tau_R]$.
We take $\{\phi_n\}\subset C^1([-d,0])$ such that
$\phi_n(-d)=0$ for all $n\in\N$ and $\phi_n\to\phi$
in $L^2(-d,0;\R)$ as $n\to\infty$. Let us define
\vspace{-0.2truecm}
\[
h_n(t, \zeta):= \int_0^{(\tau_R-t)\wedge(\zeta+d)}   e^{- (r+ \delta) \tau} g(t+\tau)\phi_n(\zeta-\tau)\,{\rm d}\tau,
\vspace{-0.2truecm}
\]
for $ (t,\zeta)\in [0, \infty) \times [-d,0)$.
It is clear that $h_n(t,\cdot)\to h(t,\cdot)$ in  $L^2(-d,0;\R)$. Moreover in step $(ii)$ we proved that
\vspace{-0.2truecm}
\[
\frac{ \partial h_n(t,\zeta)}{\partial \zeta}
= \int_0^{(\tau_R-t)\wedge(\zeta+d)}
e^{-(r+ \delta) \tau} g(t+\tau)\phi_n^{\prime}(\zeta-\tau)
\,{\rm  d}\tau.
\vspace{-0.2truecm}
\]
By applying integration by parts to the above formula we get
\vspace{-0.2truecm}
\begin{align*}
& \frac{ \partial h_n(t,\zeta)}{\partial \zeta} = %-e^{-(r+\delta)\tau}g(t+\tau)\phi_n(s-\tau)\bigg|_0^{(\tau_R-t)\wedge(s+d)}+\int_0^{(\tau_R-t)\wedge(s+d)}e^{-(r+\delta)\tau}g'(t+\tau)\phi_n(s-\tau)\,{\rm d}\tau+ \\&-(r+\delta)\int_0^{(\tau_R-t)\wedge(s+d)}e^{-(r+\delta)\tau}g(t+\tau)\phi_n(s-\tau)\,{\rm d}\tau=\\ %
g(t)\phi_n(\zeta)+\\
& \int_0^{(\tau_R-t)\wedge(\zeta+d)}
e^{-(r+\delta)\tau}g'(t+\tau)\phi_n(\zeta-\tau)\,{\rm d}\tau-(r+\delta)h_n(t,\zeta).
\vspace{-0.2truecm}
\end{align*}
By the same computations of point $(iii)$ we obtain, using the closeness of the derivative operator,
\vspace{-0.2truecm}
\begin{align*}
& L^2-\lim_{n\to\infty}\frac{ \partial h_n(t,\cdot)}{\partial \zeta}
=g(t)\phi(\cdot)+\\
& \int_0^{(\tau_R-t)\wedge(\cdot+d)}e^{-(r+\delta)\tau}
g'(t+\tau)\phi(\cdot-\tau)\,{\rm d}\tau
-(r+\delta)h(t,\cdot) = \\
& \frac{ \partial h(t,\cdot)}{\partial \zeta}.
\vspace{-0.2truecm}
\end{align*}
This implies that $h(t,\cdot)\in W^{1,2}(-d,0;\R)$,  which gives the
first part of the claim.
For the second part, by the definition of
$A^*$ and by the equation just above,
\vspace{-0.2truecm}
\begin{align*}
& \|A^*(g(t),h(t)) \|_{M_2}\le
|\mu g(t)+h(t,0)|+ \\
& \left\|-\frac{ \partial h(t,\cdot)}{\partial \zeta}
+   g(t)\phi(\cdot)\right\|_{L^2}   \\
& \le  \mu \|g\|_\infty+ C
\left(\|h\|_\infty+\|g'\|_\infty \|\phi\|_{L^2}\right)
%\le
%\|g\|_\infty (1+\|\phi\|_{L^2})+ \|h\|_\infty
%\|g'\|_\infty\|\phi\|_{L^2(-d,0;\R)}+g_{max}\|\phi\|_{L^2(-d,0;\R)}
\vspace{-0.2truecm}
\end{align*}
for some $C>0$. This immediately gives:\\
$\sup_{t \in [0, \tau_R)}\|(g(t),h(t,\cdot)) \|_{\mathcal{D}(A^*)}
< \infty$.

Finally, the fact that the couple $(g,h)$ satisfy
\eqref{CAUCHY_PROBLEM_g_h} simply follows by the above
computations and by the definition of $A^*$.
%Therefore $(g,h)$ is a strong solution (and also a mild solution)
%of \eqref{CAUCHY_PROBLEM_g_hnew}.

\smallskip

\textbf{Proof of (v).}
By \eqref{CAUCHY_PROBLEM_g_h} and \eqref{DEF_g_h} we have
\vspace{-0.2truecm}
\begin{equation*}
0= g'(\tau_R) \ne \lim_{t \rightarrow \tau_R^-}\left[ \beta g(t)-h(t,0)+1\right]=1.
\vspace{-0.2truecm}
\end{equation*}
\end{proof}

\begin{proof} of Proposition 3.4.
\\
Thanks to Proposition \ref{equiv_stoch}
the second equation in \eqref{DYNAMICS_WEALTH_LABOR_INCOME}
and equation \eqref{INFINITE_DIMENSIONAL_STATE_EQUATION} with initial
datum $t=0$, are equivalent and we have the identification
\vspace{-0.2truecm}
\begin{equation*}
 (X_0(s),X_1(s)(\zeta))=\left(y(s),y(s+\zeta)\right),
 \quad s\ge 0,\; \zeta\in[-d,0),
\vspace{-0.2truecm}
\end{equation*}
as stated in \eqref{eq:Yy}. It is then immediate to see that the formulations for the Human Capital \eqref{1} and \eqref{1bis} are equivalent.

We now prove that representation \eqref{1bis} holds true.
We want to apply It\^o's formula to compute the differential of the r.h.s. of \eqref{1bis} for $s\in [0,\tau_R)$,  as for $s\ge \tau_R$ \eqref{1bis} is clearly true.
This is possible by Lemma \ref{LEMMA_EXISTENCE_g_h}-(iii).
In this case we can apply
Ito's formula \cite[Proposition~1.165]{FABBRI_GOZZI_SWIECH_BOOK}
to the process
$\langle (g(s),h(s)),(X_0(s),X_1(s))\rangle{_{M_2}}$, thus obtaining,
for $s \in  [0,\tau_R)$,
\begin{align}
\label{c}
&{\rm d} \langle \big(g(s),  h(s)  \big), \big( X_0 (s), X_1(s)\big)
\rangle_{M_2}
\notag\\
&=
\langle (g^{\prime}(s),h'(s) ),
( X_0 (s), X_1(s))\rangle_{M_2} \, {\rm d}s
 \\
&+
\langle  A^* \big(g(s),  h(s)  \big),  (X_0(s), X_1(s))
    \rangle_{M_2}\,  {\rm d}s
+   g(s) X_0(s)   \sigma_y^\top {\rm d} Z(s).
\notag
\end{align}
By Lemma \ref{LEMMA_EXISTENCE_g_h}-(iv) the pair $(g,h)$ is the unique solution to (\ref{CAUCHY_PROBLEM_g_hnew}).
Thus we can rewrite \eqref{c} as
\begin{align}
\label{cbis}
& {\rm d} \langle \big(g(s),  h(s)  \big),
 X (s)   \rangle_{M_2} =
 \notag\\
 &
\Big[[(\beta + \mu_y) g(s)-1]X_0(s)
+ (r+ \delta)\langle h(s),X_1(s)\rangle\Big] {\rm d}s
 \\ &
+    g(s) X_0(s)   \sigma_y^\top {\rm d} Z(s). \notag
\end{align}
Let us now define
\small{\begin{align}
\label{d}
\overline{H}(s)&:= \xi(s) \big( g(s) X_0(s) +  \langle h(s)
, X_1(s)\rangle\big)\\
& =\xi(s) \<(g(s),h(s)),X(s)\>_{M_2}.\notag
\end{align}
}

Recalling %the
SDE  \eqref{DYN_STATE_PRICE_DENSITY} satisfied by the pre-death state-price density $\xi$, we can apply the standard It\^{o} formula to $\overline H$ (seen as the product of two one-dimensional processes) and use \eqref{cbis} to obtain:
\vspace{-0.1cm}
\begin{align}
\label{Hbar}
&{\rm d}\overline H(s) =
{\rm d}
\left[\xi(s) \<(g(s),h(s)),X(s)\>_{M_2}\right] \\
%\xi(s)\Big(
%- \big( r+\delta \big)\<(g(s),h(s)),X(s)\>_{M_2}
%\notag
%\\
% +&(\beta + \mu_y) g(s ) X_0(s) -X_0(s)
%+ (r+ \delta) \langle h(s) ,  X_1(s)\rangle
%-      g(s) X_0(s)  \kappa^\top  \sigma_y   \Big) {\rm d}s
%\notag\\
% & + \Big(   -\xi(s)
%\<(g(s),h(s)),X(s)\>_{M_2}  \big)
% \kappa^\top
%  +\xi(s)  g(s)  X_0(s) \sigma_y^\top \Big) {\rm d} Z(s)
%  \notag\\
%%
&= -\xi(s)  X_0(s) {\rm d}s
+ \left[  -  \overline H(s)\kappa^\top+    \xi(s)  g(s) X_0(s)
\sigma_y^\top\right]{\rm d} Z(s), \notag
\end{align}
\vspace{-0.1cm}
where we used \eqref{eq:defbeta}.
%$-(r+\delta) + \beta + \mu_y -\kappa^\top \sigma_y =0$.
Let us now fix $0 \le s \le \tau_R$.
From \eqref{Hbar} we infer, integrating on $[s,\tau_R]$,
the following result:
%\vspace{-0.1cm}
\begin{align}
\label{ITO_INT_FIN_RET}
& \overline  H(\tau_R)
-\overline H(s)
=- \int_s^{\tau_R }   \xi(u)  X_0(u)\, {\rm  d}u +\\
& + \int_s^{\tau_R }  \left[  - \overline H(u) \kappa^\top
 +\xi(u) g(u)X_0(u)  \sigma_y^\top \right ]    {\rm d} Z(u).\notag
\end{align}
Now observe that $g$ and $h$ are bounded and $\xi$ and $X$ are
$p$-integrable processes for any $p\ge 1$,
see e.g. \cite[Theorem 1.130]{FABBRI_GOZZI_SWIECH_BOOK}. Hence,
we can take the conditional mean in (\ref{ITO_INT_FIN_RET}) and, using the fact that the stochastic integral appearing there is a
%zero-mean
martingale starting at zero, we obtain
\vspace{-0.2truecm}
 \begin{equation}
 \label{ITO_INT_FIN_RET_I}
\mathbb E \left[  \overline H(\tau_R)  \mid  \mathcal F_s \right]
-\overline H(s)
= -   \mathbb E \left[ \int_s^{\tau_R}   \xi(u)  X_0(u) \,{\rm d}u
\mid  \mathcal F_s \right] .
\end{equation}
Since, by the definition of $g$ and $h$, $(g(\tau_R),h(\tau_R))=0\in M_2$
(see equation (\ref{CAUCHY_PROBLEM_g_h})) we get, using \eqref{d},
\vspace{-0.2truecm}
\begin{equation*}
\overline H( \tau_R )=0, \qquad \mathbb P-a.s..
\vspace{-0.2truecm}
\end{equation*}
Therefore \eqref{ITO_INT_FIN_RET_I} becomes
\vspace{-0.2truecm}
\begin{equation}
\label{e}
\overline H(s) = \mathbb E \left[\int_s^{\tau_R}
\xi(u)  X_0(u) \,{\rm d}u\mid \mathcal F_s \right].
\vspace{-0.2truecm}
\end{equation}
Equations \eqref{HC}, \eqref{d}, and \eqref{e} then yield the desired result.
\end{proof}

\begin{proof} of Lemma 4.3.
\\
%\begin{itemize}
\textbf{(i).}
We take $t<\tau_R$ as the case $t \ge \tau_R$ is immediate from
the argument below.

Since the function $g'$ has a discontinuity in $\tau_R$
(and thus we can not apply Ito's formula directly on the time interval
$[t,+\infty)$), we consider separately the case $s\in[t,\tau_R)$ and
$s\in[\tau_R,+\infty)$.

When $s\in[\tau_R,+\infty)$, we have
${\rm d} \overline {\Gamma}(s)={\rm d} W(s)$
and by the first equation in \eqref{DYNAMICS_WEALTH_LABOR_INCOME},
the identity \eqref{eq:GammaProcessSDE} immediately follows, recalling
that $g\equiv0$ on $[\tau_R,+\infty)$.

When $s\in[0,\tau_R)$, we have that
\vspace{-0.2truecm}
\begin{equation*}
{\rm d} \overline {\Gamma}(s)={\rm d} W(s)+ {\rm d} \left( g(s )X_0(s) + \langle    h(s)  , X_1(s) \rangle \right),
\vspace{-0.2truecm}
\end{equation*}
and proceeding as in the proof of Proposition
\ref{PROPOSITION_EVALUATION_DELAYED_LABOR_INCOME} we can write
\vspace{-0.1truecm}
\begin{align*}
&{\rm d} \<(g(s),h(s)),X(s)\>  =\\
& = \left[[(\beta + \mu_y)g(s)-1]X_0(s)
+ (r+ \delta)\langle h(s),X_1(s)\rangle
\right] \,{\rm d}s\\
&
+    g(s) X_0(s)   \sigma_y^\top \, {\rm d} Z(s).
\vspace{-0.2truecm}
\end{align*}
Using the dynamics of $W$ given by the first equation
in \eqref{DYN_W_X_INFINITE_RETIREMENT_II} and
\eqref{DEF_KAPPA} we get
\eqref{eq:GammaProcessSDE}. This proves $(i)$.

\textbf{(ii).}
This proof, given (i) above,
is completely analogous to the proof of \cite[Lemma 4.5(ii)]{BGP}
and we omit it for brevity.
\end{proof}

\begin{proof} of Proposition 5.1.
\\
For the proof of the result we appeal to \cite[Theorem 5.1]{BGP}
fixing $(X_0,X_1)\equiv 0$, for example setting the initial condition
for $X$ at $0$, i.e. $x=0\in M_2$ in their notation. When
we apply this result we just have to pay attention to the different initial
 time of the problem. This is not a big issue since the problem is
 autonomous and it is sufficient to rescale time. We thus obtain that
  the value function associated with Problem \ref{pb2} is as in \eqref{Vfh}.
 \end{proof}

\begin{proof} of Proposition 6.4.
\\
Recall that, by definition, $\mathcal{H}_{++}^{0, \tau_R}$ is the set where $\Gamma$ is strictly positive. Thanks to the linearity of $\Gamma$, the function $\overline v$ is twice continuously differentiable in $(w,x)$ and continuously differentiable in $t$ (see Lemma \ref{LEMMA_EXISTENCE_g_h}-(iii)).
The derivatives that appear in the Hamiltonian are easily computed:
\begin{align*}
&\partial_t\overline v(t,w,x)= \frac{\gamma}{1-\gamma} \Gamma^{1-\gamma}
F^{\prime}(t)F^{\gamma-1}(t) +
g^{\prime}(t)x_0 F(t)^{\gamma}\Gamma^{-\gamma} + \\
&  + \langle
h'(t),x_1\rangle
F(t)^{\gamma}\Gamma^{-\gamma},
\\
&\partial_w\overline v(t,w,x)=F(t)^{\gamma}\Gamma^{-\gamma},
\,
\partial_{x_0}\overline v(t,w,x)= F(t)^{\gamma}\Gamma^{-\gamma}g(t),
\\
&
\partial_{x_1}\overline v(t,w,x)= F(t)^{\gamma}\Gamma^{-\gamma} h(t),
\, \notag \\
& \partial_{ww}\overline v(t,w,x)= - \gamma F(t)^{\gamma}\Gamma^{-\gamma-1},
\\
& \partial v_{wx_0}\overline v(t,w,x)=   - \gamma F(t)^{\gamma}
\Gamma^{-\gamma-1} g(t),\\
&
\partial v_{x_0 x_0}\overline v(t,w,x)= - \gamma F(t)^{\gamma}
\Gamma^{-\gamma-1}  g(t)^2.
\end{align*}
Thanks to Lemma \ref{LEMMA_EXISTENCE_g_h}-(iv),
also requirement \ref{regularity} in
Definition \ref{DEF_SUPERSOLUTION_INF_RET} is satisfied.
Moreover, by Hypothesis \ref{hp:betarho} we get
$\partial_w\overline v>0$ and $\partial_{ww}\overline v<0$ on
$\mathcal{H}_{++}^{0, \tau_R}$.
Therefore we can consider the simplified form \eqref{new_HJB} for the HJB equation.
Using Proposition \ref{adj} we compute:
\begin{align*}
\langle x,A^*\partial_x\overline v (t,w,x)  \rangle_{M_2} &= F(t)^{\gamma} \Gamma^{-\gamma}  \langle x, A^*(g(t),  h(t,\cdot)) \rangle_{M_2}
%\\ & = F(t)^{\gamma} \Gamma^{-\gamma} \left[ x_0  h(t,0)  + x_0\mu_y g(t) +\langle x_1, g(t)\phi- \frac{ \partial h(t, \cdot)}{\partial s}\rangle\right].
\end{align*}
Substitute the above expression for   $\langle x,A^*\partial_x \overline v(t,w,x)  \rangle_{M_2}$ and $\overline v$ with all its derivatives in \eqref{new_HJB}. Now, after multiplying both terms by $F(t)^{-\gamma}\Gamma^{\gamma}$, we get:
\begin{align}\label{HJB_BEFORE_T}
& -\frac{\gamma}{1-\gamma} \Gamma\frac{F^{\prime}(t)}{F(t)}
- \<(g^{\prime}(t),h'(t)),x\>_{M_2}
=(r+\delta) w +x_0  \notag\\
&
+\langle x, A^*(g(t),h(t)) \rangle_{M_2}
 +  \frac{\gamma}{1-\gamma} e^{-\frac{\rho+\delta}{\gamma}t}
 F(t)^{-1} \Gamma(1+\delta k^{-b}  ) \notag \\
& + \frac{\kappa^\top \kappa}{2\gamma }\Gamma
-x_0 g(t)\kappa^\top \sigma_y.
\end{align}
Thanks to \eqref{CAUCHY_PROBLEM_g_hnew},
equality (\ref{HJB_BEFORE_T}) can be rewritten as
\begin{align*}
& \frac{\gamma}{1-\gamma} \Gamma\frac{F^{\prime}(t)}{F(t)} +
\left((\beta+\mu_y) g(t)- 1\right)x_0 +
(r + \delta)\langle h(t) ,x_1\rangle
\\
&+(r+\delta) w + x_0 + \frac{\gamma}{1-\gamma}
e^{-\frac{\rho+\delta}{\gamma}t} F(t)^{-1} \Gamma(1+\delta k^{-b}  )
+\notag \\
 & \frac{\kappa^\top \kappa}{2\gamma }\Gamma  -x_0g(t)
 \kappa^\top \sigma_y  = 0,
\end{align*}
which, using also the definition of $\beta$ in \eqref{eq:defbeta}
simplifies as
 \begin{align}\label{EQ_F}
& \left(\frac{\gamma}{1-\gamma} \frac{F^{\prime}(t)}{F(t)} + (r + \delta) +  \frac{\gamma}{1-\gamma} e^{-\frac{\rho+\delta}{\gamma}t} F(t)^{-1} (1+\delta k^{-b}  ) + \right. \notag \\
&\left. + \frac{\kappa^\top \kappa}{2\gamma }\right)\Gamma  =0,
\end{align}
Recalling \eqref{fbnu} and \eqref{F.} we compute the derivative of $F$
and get that \eqref{EQ_F} is satisfied in $\mathcal{H}_{++}^{0, \tau_R}$.
\end{proof}

\begin{proof} of Lemma 6.6.
\\
The result is trivial when $\gamma \in (1, \infty)$.
Let then take $\gamma \in (0,1)$.
We have, by \eqref{EQ_GUESS_V} and \eqref{Gamma_bar},
\begin{align*}
&\mathbb{E} \left[ \sup_{s \in [t,\tau_R]}\overline v
\left(s, (W^{w,x}(s;t, \pi), X^x(s;t))\right)\right] \\
 &= \mathbb{E} \left[ \sup_{s \in [t,\tau_R]}F^{\gamma}(s)
 \frac{\overline \Gamma^{1-\gamma}(s) }{1-\gamma}\right]\\
& \le \max_{s \in [t,\tau_R]}F^{\gamma}(s) \mathbb{E}
 \left[ \sup_{s \in [t,\tau_R]}\frac{\overline \Gamma^{1-\gamma}(s)}
 {1-\gamma}\right] <\infty.
\end{align*}
As from Lemma \ref{lm:Gammabar},
the time evolution of the total wealth process $\overline \Gamma(\cdot)$,
on $[t, \tau_R]$, is described by the SDE \eqref{eq:GammaProcessSDE}.
Theorem 1.130 in \cite{FABBRI_GOZZI_SWIECH_BOOK}
implies $\mathbb{E} \left[\sup_{s \in [t, \tau_R]}\overline \Gamma (s) \right]
 < \infty$ which gives the claim.
\end{proof}

\begin{proof} of Proposition 6.7.
\\
Recalling notation \eqref{Gamma_bar}, let us call
\vspace{-0.2truecm}
\begin{equation}
\label{tauN}
\tau_N=\inf\left\{{s\ge t}: \;\overline\Gamma(s)\le 1/N\right\}.
\vspace{-0.2truecm}
\end{equation}	
For $N$ sufficiently large we have $\tau_N>0$. Moreover notice that $\tau_N \rightarrow  \tau_t$, $\mathbb{P}$-a.s., so, still
$\mathbb{P}$-a.s., $\tau_N\wedge \tau_R \rightarrow  \overline\tau_t$.
Let us fix any such $N$. We apply It\^{o}'s formula to the process $\overline v\big(s,\mathcal{X}(s; \pi)\big)$ for $s\in[t,\tau_N \wedge \tau_R]$.
This is possible since, in this interval, we know that the process
$(s,\mathcal{X}(s;t,\pi))$ belongs to the set $\{\Gamma(s,w,x)\ge 1/N\}$
where $\overline{v}$ and its derivatives are
bounded on bounded sets.\footnote{Indeed,
as highlighted in the proof of \cite[Proposition 4.11]{BGP},
to apply Ito's formula here one should slightly adapt the proof of  \cite[Proposition 1.164]{FABBRI_GOZZI_SWIECH_BOOK}.}

We set, for brevity, $\mathcal{X}(\cdot;t, \pi):=\mathcal{X}(\cdot)$,
$W(\cdot;t, \pi):=W(\cdot)$,
$X(\cdot;t):=X(\cdot)$,
and we apply \cite[Proposition 1.164]{FABBRI_GOZZI_SWIECH_BOOK} to obtain
\vspace{-0.3truecm}
\begin{align*}
&\overline v\left(\tau_N \wedge \tau_R,
\mathcal{X}(\tau_N \wedge\tau_R)\right)
- \overline v(t,w,x)= \\
& =
\int_t^{\tau_N \wedge \tau_R} {\rm d}s \,
\Big\{\partial_s\overline v \big(s,\mathcal{X}(s))
+ \partial_{x_0 w} \overline v \big(s,\mathcal{X}(s)\big)X_0(s)
\theta^\top(s) \sigma  \sigma_y
\\
&+ \partial_w  \overline v \big(s,\mathcal{X}(s)\big)
\big[ (r+ \delta) W(s)+ \theta^\top(s) (\mu-r \mathbf 1)  + \\
&  X_0(s) - c(s) - \delta B(s) \big]  +
\langle X(s),A^*\partial_x\overline v \big(s,\mathcal{X}(s)\big)  \rangle_{M_2} + \\
& \frac{1}{2}\partial_{ww} \overline v \big(s,\mathcal{X}(s)\big)  \theta^\top \sigma  \sigma^\top \theta
+\frac{1}{2} \partial _{x_0 x_0} \overline v \big(s,\mathcal{X}(s)\big) \sigma_y^\top \sigma_y X^2_0 (s)
\Big\}\, +
\\
&  \int_t^{\tau_N \wedge \tau_R} \left[ \partial_{w} \overline v\big(s,\mathcal{X}(s)\big) \theta^\top(s) \sigma +\partial_{x_0} \overline v\big(s,\mathcal{X}(s)\big)  X_0 (s) \sigma_y^\top \right] \, {\rm d}Z(s).
\vspace{-0.3truecm}
\end{align*}

Since the function $\overline v$ solves the HJB equation (\ref{HJB.}), by the definition of $\mathbb H^{fh}_{cv}, \mathbb H^{fh}_{\max} $ in (\ref{DEF_H_CV_FIN_HOR}) and (\ref{DEF_HAM_max FIN_HOR}) respectively, we can rewrite the above equality as
\vspace{-0.3truecm}
\begin{multline}\label{EXP_VERIFICATION_FIN_RET}
\begin{split}
&\overline v
\big(\tau_N \wedge \tau_R,\mathcal{X}(\tau_N \wedge \tau_R)\big)
- \overline v\big(t,w,x\big) +\\
& +  \int_{t}^{\tau_N \wedge \tau_R} e^{-(\rho+ \delta) s }
\Big( \frac{ c(s)^{1-\gamma}}{1-\gamma}
+ \delta \frac{\big(k B(s)\big)^{1-\gamma}}{1-\gamma}\Big)\, {\rm d}s=
\\&
% \textcolor[rgb]{1.00,0.00,0.00}
{\int_t^{\tau_N \wedge \tau_R}
 \left[
-\mathbb H^{fh}_{\max}\big(s,X_0(s),
\partial_w\overline v,
\partial_{ww} \overline v,
\partial_{wx_0}\overline v\big) \right.}
\\
&
%\textcolor[rgb]{1.00,0.00,0.00}
{\left.+\mathbb H^{fh}_{cv}
\big(s,X_0(s),\partial_w\overline v,
\partial_{ww} \overline v,
\partial_{wx_0}\overline v
;\pi(s)
\big)  \right] \,{\rm d}s}
\\
&+ \int_t^{\tau_N \wedge \tau_R}
\left[  (\partial_{w} \overline v)
\theta^\top(s) \sigma
+(\partial_{x_0} \overline v)
X_0 (s)\sigma_y^\top \right] \, {\rm d}Z(s).
\end{split}
\vspace{-0.2truecm}
\end{multline}
Let us now take the expected value on both sides of equation \eqref{EXP_VERIFICATION_FIN_RET}.
The stochastic integral is a martingale
(by the definition of $\tau_N$ in \eqref{tauN}, by the
explicit form of $\overline v$ and its derivatives,
and by the integrability of $\theta$ and $X$) so its expectation is $0$.
The integrand at the left hand side has always finite expectation:
when $\gamma \in (0,1)$ by the integrability properties of $c$ and $B$,
while for $\gamma>1$, since we assumed $J^{fh}(t,w,x;\pi)>-\infty$.
We then get that also the expectation of the first term
of the right hand side (which exists since the integrand is negative)
is finite. Hence, using the same notation as in the statement:
\vspace{-0.2truecm}
\begin{align}
\label{new:e}
&\mathbb{E}\left[\overline v\big(\tau_N \wedge \tau_R,\mathcal{X}(\tau_N \wedge \tau_R)\big)\right]-
\overline v\big(t,w,x\big) \\
&
+ \mathbb{E} \left[
\int_{t}^{\tau_N \wedge \tau_R} e^{-(\rho+ \delta) s }
\Big( \frac{ c(s)^{1-\gamma}}{1-\gamma}+ \notag    \delta \frac{\big(k B(s)\big)^{1-\gamma}}{1-\gamma}\Big)\, {\rm d}s\right]
\notag
\\
&= -\mathbb{E} \bigg[\int_t^{\tau_N \wedge \tau_R}
 \Big[
\mathbb H^{fh}_{\max}(s,X_0(s),
\partial_w\overline v,
\partial_{ww} \overline v,
\partial_{wx_0}\overline v)
\notag
\\
&-\mathbb H^{fh}_{cv}\big(s,X_0(s),
\partial_w\overline v,
\partial_{ww} \overline v,
\partial_{wx_0}\overline v
;\pi(s)\big)  \Big] \,{\rm d}s\bigg].
\notag
\end{align}
Now let $N \rightarrow \infty$.
The first term of the left hand side converges thanks to Lemma \ref{LEMMA_LIMIT_AT_INFTY_GUESS_VALUE_FUNCTION_INFINITE_RETIREMENT}, the continuity of $\overline v$, and dominated convergence.
The last term of the left hand side
converge to a finite number:
when $\gamma \in (0,1)$ by the integrability properties of $c$ and $B$,
while for $\gamma>1$, since we assumed $J^{fh}(t,w,x;\pi)>-\infty$.
Moreover the right hand side
converge thanks to the constant sign of the integrands
and monotone convergence: the limit is finite
since the left hand side is finite. Hence
\vspace{-0.2truecm}
\begin{align}
\label{EXP_VERIFICATION_FIN_RET.}
%\begin{split}
&\mathbb{E}\left[\overline v\big(\overline\tau_t,
\mathcal{X}(\overline\tau_t)\big)\right]-
\overline v\big(t,w,x\big)
+ \\
& \mathbb{E} \left[
\int_{t}^{\overline\tau_t} e^{-(\rho+ \delta) s }
\Big( \frac{ c(s)^{1-\gamma}}{1-\gamma}+
\delta \frac{\big(k B(s)\big)^{1-\gamma}}{1-\gamma}\Big)\, {\rm d}s\right]
\notag
\\
&= -\mathbb{E} \left[\int_t^{\overline\tau_t}
 \left[
\mathbb H^{fh}_{\max}\big(s,X_0(s),
\partial_w\overline v,
\partial_{ww} \overline v,
\partial_{wx_0}\overline v \big) \right.\right.
\\
%[1mm]
\notag
&\left.\left.
  -\mathbb H^{fh}_{cv}\big(s,X_0(s),
\partial_w\overline v,
\partial_{ww} \overline v,
\partial_{wx_0}\overline v
;\pi(s)\big)  \right] \,{\rm d}s\right].
%\end{split}
\vspace{-0.2truecm}
\end{align}
The claim now follows rearranging the terms and observing that
\vspace{-0.2truecm}
\begin{align*}
& J^{fh}(t,w, x;\pi)
=\\
&\mathbb{E}\left[ \int_{t}^{\overline \tau_t}
e^{-(\rho+ \delta)s} \Big( \frac{c(s)^{1-\gamma}}{1-\gamma} + \delta \frac{\big(k B(s)\big)^{1-\gamma}}{1-\gamma}\Big) \,{\rm d}s\right]
+\\
&  \mathbb{E}\left[ \overline v\big(\overline \tau_t, \mathcal{X}
(\overline \tau_t)\big)\right].
\vspace{-0.2truecm}
\end{align*}
The above is obvious when $\overline \tau_t= \tau_R$
by the form of $\overline v$ in
Proposition \ref{PROPOSITION_GUESS_SOLVES_HJB_FIN_HOR}.
When $\overline \tau_t< \tau_R$ (which can happen only for
$\gamma \in (0,1)$ due to our assumptions), it must be, by
\eqref{EQ_GUESS_V},
$\mathbb E \left[\overline v\big(\overline \tau_t,
\mathcal{X}(\overline\tau_t)\big)\right] = 0$,
and, by Lemma \ref{lm:Gammabar}-(ii),
\vspace{-0.2truecm}
\begin{equation*}
\mathbb I_{\{\overline \tau<s\}}(\omega) c(s,\omega)=0, \quad \mathbb I_{\{\overline\tau<s\}}(\omega)B(s,\omega)=0,
\vspace{-0.2truecm}
\end{equation*}
holds $ ds\otimes \P-a.e. \; in \; [0,\tau_R] \times \Omega$,
and the claim still follows.
\end{proof}

\begin{proof} of Corollary 6.8.
\\
For $\gamma \in (0,1)$ the positivity of
the integrand in \eqref{EXP_VER_INF_RET_III} implies,
for every
$(t,w,x)\in {\mathcal{H}}_{++}^{0,\tau_R}$ and $\pi\in \Pi(t,w,x)$,
$\overline v(t,w,x) \ge J^{fh}(t,w,x;\pi)$.
Then the claim follows by the definition of $V^{fh}$.
%$V^{fh}(t,w,x):= \sup_{\pi\in \Pi(t,w,x)}J^{fh}(t,w,x;\pi)$.
The same argument works for $\gamma >1$ if we prove that there exists
a strategy in $\Pi(t,w,x)$ such that $J^{fh}(t,w,x;\pi)>-\infty$.
Such strategy is the one given in
\eqref{EQ_FEEDBACK_STRATEGIES_INF_RET} below.
\end{proof}

\begin{proof} of Lemma 31.
\\
Substituting \eqref{EQ_FEEDBACK_STRATEGIES_INF_RET} into the first equation in \eqref{eq:newCL} we obtain
\begin{align}
\label{CLOSED_LOOP_W_FINITE_HOR}
& {\rm d} W^*_f(s) = \Big[ \Gamma^*(s)\left[  \frac{|\kappa|^2}{\gamma }   - f(s)^{-1}\big(1+ \delta k^{-b}\big)  \right] +\\
& W^*_f(s)(r+\delta) - \kappa^\top \sigma_y  g(s) X_0(s) + \notag\\ & X_0(s) \Big]\,{\rm d}s+ \left[ \frac{\Gamma^*(s)}{ \gamma } \kappa^\top - g(s) X_0(s)   \sigma_y^\top \right]\, {\rm d}Z(s).\notag
\end{align}
Since
%\begin{equation*}
${\rm d}\Gamma^*(s)={\rm d}W_f^*(s)+{\rm d} \Big( g(s)X_0(s) + \langle    h(s)  , X_1(s) \rangle \Big),
%\end{equation*}
$
recalling \eqref{cbis} and the definition of $\beta$ in
\eqref{eq:defbeta}, we immediately get the claim.
\end{proof}

\begin{proof} of Theorem 6.12.
\\
%[Proof of Theorem \ref{THM_VERIFICATION_THEOREM_INF_RET}]
First take $(t,w,x) \in \partial\mathcal{H}_+^{0 , \tau_R}$.
By equation \eqref{DYN_GAMMA*_PROPOSITION} we thus have that
for every $s \in[t, \tau_R]$, $\Gamma^*(s)=0$, $\P$-a.s..
This in turns implies, by \eqref{EQ_FEEDBACK_STRATEGIES_INF_RET}, that
\vspace{-0.2truecm}
$$
\overline c_f \equiv 0,\qquad
\overline B_f \equiv 0,\qquad
\overline \theta_f \equiv - g(t) X_0(s) (\sigma^\top)^{-1} \sigma_y.
\vspace{-0.2truecm}
$$
It follows from Lemma \ref{lm:Gammabar}-(ii) that this is the
only admissible strategy, hence it must be optimal and,
clearly, $V^{fh}=\overline v$ in such boundary points.

Now take $(t,w,x) \in \mathcal{H}_{++}^{0, \tau_R}$.
First we observe that,
%$\overline{\pi}_f:=(\overline c_f, \overline B_f,\overline \theta_f)$ is an admissible strategy. Indeed
by Lemma \ref{36}, $\Gamma^*(\cdot)$ is a stochastic exponential, thus $\P$-a.s. strictly positive for any strictly positive initial condition $\Gamma^*(t)=\Gamma(t,w,x)$. Hence,
the constraint in (\ref{CONSTRAINT_FIN_RET}) is always satisfied
with strict inequality and that $(\overline c_f,\overline B_f)$ are strictly positive
by \eqref{EQ_FEEDBACK_STRATEGIES_INF_RET}.
This implies that $\overline{\pi}_f$ is admissible and that
$J^{fh}(t,w,x;\overline{\pi}_f)>-\infty$ when $\gamma >1$
and the fundamental identity \eqref{EXP_VER_INF_RET_III} can be used.

Now, since the feedback map \eqref{EQ_DEF_FEEDBACK_MAP} is obtained taking the maximum points of the Hamiltonian
the fundamental identity \eqref{EXP_VER_INF_RET_III} becomes
\vspace{-0.3truecm}
\begin{equation*}
\overline v(t,w,x)=J^{fh}\left(t,w,x;\overline\pi_f\right).
\vspace{-0.2truecm}
\end{equation*}
Hence, Corollary \ref{cr:FINITENESS_VALUE_FUNCTION} and the definition of the value function yields
\vspace{-0.2truecm}
$$
V^{fh}(t,w,x)\le \overline v(t,w,x)=J^{fh}\left(t,w,x;\overline\pi_f\right)\le V^{fh}(t,w,x)
\vspace{-0.2truecm}
$$
which immediately gives $V^{fh}(t,w,x)=J^{fh}\left(t,w,x;\overline\pi_f\right)$, hence optimality of $\overline \pi_f$.

We prove uniqueness. When $(t,w,x)\in \partial\mathcal{H}_+^{0, \tau_R}$ the claim follows from Lemma \ref{lm:Gammabar}-(ii).
Let $(t,w,x)\in \mathcal{H}_{++}^{0,\tau_R}$.
Since $\overline v=V^{fh}$, an optimal strategy $\pi$ at
$(t,w,x)$ must satisfy
$\overline v(t,w,x)=J^{fh}(t,w,x;\pi)$, which implies, substituting in \eqref{EXP_VER_INF_RET_III}, that the integral in \eqref{EXP_VER_INF_RET_III} is zero. This implies that, on $[t,\overline \tau_t]$ we have $\pi=\overline\pi_f$, $dt \otimes \P$-a.e.. This provides uniqueness, as, for $\overline\pi_f$
we have $\overline \tau_t=\tau_R$.
\end{proof}

\begin{proof} of Corollary 6.13.
\\
Lemma \ref{lm:Gammabar} implies that equality holds
on $\partial \mathcal{H}_+^{0, \tau_R}$.
Let $(t,w,x)\in \mathcal{H}_{++}^{0, \tau_R}$. From the definition of
$V$, breaking the integral at $\tau_R$ we see that $V\le V^{fh}$.
Moreover, the strategy found concatenating the optimal strategies of the
Problems 3 and 4 is admissible for Problem 2, hence $V\ge V^{fh}$.
\end{proof}

\begin{proof}  of Theorem 7.1.
\\
The result is a direct consequence of Proposition \ref{PROPOSITION_MAIN_INF_HOR}, Theorem \ref{th:VERIFICATION_THEOREM_INF_RET} and Corollary \ref{cor2}.
\end{proof}

%
%
%\begin{figure}[h!]
%	\centering
%	\resizebox*{12cm}{8cm}{\includegraphics{Figure1.png}}\\
%	\caption{\mygreen  Expected optimal allocation to risky assets as a fraction of total wealth, $E[\theta^*_f/\Gamma^*]$. Time $0$ corresponds to age $25$, time $\tau_R=40$ to the retirement age of $65$. Parameter values are as follows: $\rho=0.04$, $r=0.01$, $\mu=0.07$, $\sigma=0.16$, $\delta=0.015$, $k=0.05$, $\gamma=5$, $K=3$,
%		$W(0)=1$, $y(0)=1.5$.
%		Example of labor income dynamics: (a1) $\phi=0$, $\mu_y=0.04$, $\sigma_y=0.1$; (a2) $\phi=0$, $\mu_y=0.04$, $\sigma_y=0.06$; (b1) $\phi=0.0075$, $\mu_y=0.0025$, $\sigma_y=0.1$;
%		(b2) $\phi=0.0075$, $\mu_y=0.0025$, $\sigma_y=0.08$. Parameter configurations (a1) and (b1) give rise to a negative hedging demand, as $\sigma_y>\gamma^{-1}\kappa$. }
%\end{figure}

\end{document}